\documentclass{amsart}      
\usepackage{amssymb,amsmath,graphicx}
\newtoks\prt 
\numberwithin{equation}{section}
 
\newtheorem{thm}{Theorem}[section]
\newtheorem{question}[thm]{Question} 
\newtheorem{lemma}[thm]{Lemma} 
\newtheorem{prop}[thm]{Proposition} 
\newtheorem{cor}[thm]{Corollary} 
\newtheorem{example}[thm]{Example} 
\newtheorem*{claim}{Claim} 

\theoremstyle{definition} 
 
\newtheorem{remark}[thm]{Remark} 
\newtheorem*{rem}{Remark}
\newtheorem{definition}[thm]{Definition}

\def\eqn#1$$#2$${\begin{equation}\label#1#2\end{equation}}

\def\F{\mathcal F}

\def\L{\mathcal L} 
\def\U{\mathcal U}

\def\diam{\operatorname{diam}} 
 
\def\ep{\varepsilon} 
 
\def\en{\mathbb N} 
\def\er{\mathbb R}

\def \reg {\partial _{\kern1pt\text{reg}}}

\def\dd{\operatorname{d}}
\def\di{\,\mbox{\rm d}}
\def\dh{\widehat{\operatorname{d}}}
\def\clu#1#2{\operatorname{clust}_{#1^{**}}(#2)}
\def\clud#1#2{\operatorname{clust}_{#1^{***}}(#2)}
\def\wde#1{\widetilde{\delta}\left(#1\right)}
\def\de#1{\delta\left(#1\right)}

\renewcommand{\Re}{\operatorname{Re}}

\newcommand{\wcc}{\operatorname{wcc}}
\newcommand{\ca}[2][]{\operatorname{ca}_{#1}\left(#2\right)}
\newcommand{\wca}[2][]{\widetilde{\operatorname{ca}}_{#1}\left(#2\right)}
\newcommand{\cont}[2][\rho]{\operatorname{cont}_{#1}\left(#2\right)}
\newcommand{\cc}[2][]{\operatorname{cc}_{#1}\left(#2\right)}
\newcommand{\wk}[2][X]{\operatorname{wk}_{#1}\left(#2\right)}
\newcommand{\wck}[2][X]{\operatorname{wck}_{#1}\left(#2\right)}
\newcommand{\wscl}[1]{\overline{#1}^{w^*}}
\newcommand{\upleq}{\rotatebox{90}{$\,\leq\ $}}
\newcommand{\upsim}{\rotatebox{90}{$\,\sim\ $}}
\newcommand{\sleq}{\;\overset{\mbox{\normalsize$<$}}{{\mbox{\small$\sim$}}}\;}

\newcommand{\bchi}{\mbox{\Large$\chi$}}

\begin{document}

\title{Quantitative Dunford-Pettis property}
\author{Miroslav Ka\v{c}ena, Ond\v{r}ej F.K. Kalenda and Ji\v{r}\'{\i} Spurn\'y}

\address{Department of Mathematical Analysis \\
Faculty of Mathematics and Physic\\ Charles University\\
Sokolovsk\'{a} 83, 186 \ 75\\Praha 8, Czech Republic}
\email{kacena@karlin.mff.cuni.cz}
\email{kalenda@karlin.mff.cuni.cz}
\email{spurny@karlin.mff.cuni.cz}


\thanks{The second author was supported in part by the grant
GAAV IAA 100190901.}
\thanks{The third author was supported in part by the grant GA\v{C}R 201/07/0388.}
%s were supported in part by the grant
%GAAV IAA 100190901  and in part by the Research Project MSM~0021620839 from the Czech Ministry of Education.}

\begin{abstract} 
We investigate possible quantifications of the Dunford-Pettis property. We show, in particular, that the Dunford-Pettis property is automatically quantitative in a sense. Further, there are two incomparable mutually dual stronger versions of a quantitative Dunford-Pettis property. We prove that $L^1$ spaces and $C(K)$ spaces posses both of them. We also show that several natural measures of weak non-compactness are equal in $L^1$ spaces.
\end{abstract}
\maketitle

%%%%%%%%%%%%%%%%%%%%%%%%%%%%%%%%%%%%%%%%%%%%%%%%%%%%

\section{Introduction}

A Banach space $X$ is said to have the \emph{Dunford-Pettis property} if for any Banach space $Y$ every 
weakly compact operator $T:X\to Y$ is completely continuous. Let us recall that $T$ is \emph{weakly compact} if the image by $T$ of the unit ball of $X$ is relatively weakly compact in $Y$. Further, $T$ is
\emph{completely continuous} if it maps weakly convergent sequences to norm convergent ones, or, equivalently, if it maps weakly Cauchy sequence to norm Cauchy (hence norm convergent) ones.

There are several well-known classes of Banach spaces with the Dunford-Pettis property. For example,
any Banach space whose dual has the Schur property, the space $C(K)$ of continuous functions on a compact Hausdorff space and the space $L^1(\mu)$ for any non-negative $\sigma$-additive measure have the Dunford-Pettis property. The proof of the first case is an easy consequence of the Gantmacher and the Schauder theorem and will be commented below. The other two cases are proved in \cite[Th\'eor\`eme~1]{gro} and outlined also in \cite[pp. 61--62]{JoLi}.

A complementary notion is that of the {\it reciprocal Dunford-Pettis property}. A Banach space $X$ has the reciprocal Dunford-Pettis property if for any Banach space $Y$ any completely continuous operator $T:X\to Y$ is weakly compact. In general, the classes of weakly compact operators and completely continuous operators are incomparable (the identity on $\ell_2$ is weakly compact but not completely
continuous, the identity on $\ell_1$ is completely continuous but not weakly compact). The spaces of the form $C(K)$ where $K$ is a compact Hausdorff space have both the Dunford-Pettis property (see the previous paragraph) and the reciprocal Dunford-Pettis property (see \cite[p. 153, Th\'eor\`eme~4]{gro}).

In the present paper we investigate quantitative versions of the Dunford-Pettis property. It is inspired
by a number of recent results on quantitative versions of certain theorems and properties. In particular,
quantitative versions of the Krein theorem were studied in \cite{f-krein,Gr-krein,GHM,CMR}, quantitative versions of the Eberlein-\v{S}mulyan and the Gantmacher theorem were investigated in \cite{AC-meas}, a quantitative version of James' compactness theorem was proved in \cite{CKS,Gr-james}, a quantification of weak sequential completeness and of the Schur property was addressed in \cite{wesecom,qschur}.

The main idea behind quantitative versions is an attempt to replace the respective implication by an inequality. So, in case of the Dunford-Pettis property we will try to replace the implication 
$$T\mbox{ is weakly compact} \Rightarrow T\mbox{ is completely continuous}$$
by an inequality of the form
\begin{multline*} \mbox{measure of non-complete continuity of }T \\ \leq C\cdot \mbox{measure of weak non-compactness of }T.\end{multline*}
There is a natural measure of non-complete continuity (see below) and several non-equivalent natural measures of weak non-compactness of an operator. It is rather interesting that for one of these measures
of weak non-compactness the Dunford-Pettis property is automatically quantitative but for another one it is not the case.

Non-equivalence of several measures of weak non-compactness leads us to two ways of a strengthening of the Dunford-Pettis property. We call the resulting main notions of our paper the \emph{direct} and \emph{dual quantitative Dunford-Pettis property.} They are characterized in Theorems~\ref{q-DP-direct} and \ref{q-DP-dual}. Moreover, they are mutually incomparable as witnessed by Example~\ref{protipriklad}.
Both these properties are shared by classical spaces with the Dunford-Pettis property, i.e., $C(K)$ and $L_1$ spaces. This is proved in Theorem~\ref{l1-li} where it is shown  that $\L_1$ and $\L_\infty$ spaces posses both the direct and dual quantitative Dunford-Pettis property.

We also include some results on a quantitative reciprocal Dunford-Pettis property. Since we have not
investigated this property in detail, we include only those results that naturally appear as byproducts
of our investigation of the Dunford-Pettis property and related quantities. A more detailed investigation is contained in \cite{ka-spu-stu}.

The paper is organized as follows:

Section~\ref{sec-pre} contains definitions of basic quantities used in the paper, a survey of known and easy relationships and inequalities among them and a comparison of the introduced notions in complex and real Banach spaces.

In Section~\ref{sec-easy} we collect quantitative versions of easy inclusions among four classes of operators - compact, weakly compact, completely continuous and weakly completely continuous ones.

Section~\ref{sec-wco} contains quantitative versions of two known results characterizing weakly compact operators by means of their continuity in a certain topology. 

In Section~\ref{sec-twoways} we show that the Dunford-Pettis property is automatically quantitative in a sense. We further define the above mentioned natural stronger quantitative versions of the Dunford-Pettis property, establish their characterizations and mutual duality. We also formulate there the main results on $\L_1$ and $\L_\infty$ spaces proven in the sequel.

Section~\ref{sec-schur} is devoted to the relationship of the Schur property and quantitative Dunford-Pettis properties. 

In  Section~\ref{sec-L1} we show that natural measures of weak non-compactness coincide in $L^1$ spaces.
In particular, we compute these measures explicitly.

In Section~\ref{sec-directCK} we use the results of the previous sections to prove that $C(K)$ spaces and, more generally, $\L_\infty$ spaces have the direct quantitative Dunford-Pettis property.

Section~\ref{sec-dualCK} contains a quantification of some results from the measure theory and the proof that $C(K)$ spaces (and hence $\L_\infty$ spaces) have the dual quantitative Dunford-Pettis property as well.

Section~\ref{sec-example} contains an example showing that the two quantitative versions of the Dunford-Pettis property are incomparable and that a space with the Dunford-Pettis property need not  satisfy any of the two quantitative versions.

In the last section we collect some open problems.

%%%%%%%%%%%%%%%%%%%%%%%%%%%%%%%%%%%%%%%%%%%%%%
\section{Preliminaries}\label{sec-pre}

In this section we collect basic notation and definitions of the necessary quantities. Banach spaces
which we consider can be either real or complex -- all the results are valid in both cases. However,
some of the results which we are referring to are formulated and proved only for real spaces. In the first subsection we will show a general method how these results can be transferred to complex spaces.

\subsection{Real and complex spaces}

If $X$ is a (real or complex) Banach space, we define the spaces $X^{(n)}$ for $n\in\en\cup\{0\}$ as follows:
\begin{itemize}
	\item $X^{(0)}=X$,
	\item $X^{(n)}=(X^{(n-1)})^*$ for $n\in\en$.
\end{itemize}
Further, if $X$ is a complex Banach space, we denote by $X_R$ the real version of $X$, i.e., the same space considered over $\er$ (we just forget multiplication by imaginary numbers). 

Then the spaces $X^{(n)}$, $(X^{(n)})_R$ and $(X_R)^{(n)}$ can be related as described in the following
proposition whose straightforward proof we omit.

\begin{prop}\label{recom} Let $X$ be a complex Banach space. For each $n\in\en\cup\{0\}$ let $\iota_n:X^{(n)}\to(X^{(n)})_R$ be the identity mapping. Further, let us define mappings $\psi_n:(X^{(n)})_R\to(X_R)^{(n)}$ by induction as follows:
\begin{itemize}
	\item $\psi_0$ is the identity of $X_R$.
	\item $\psi_n(f)(x)=\Re \iota_n^{-1}(f)(\iota_{n-1}^{-1}(\psi_{n-1}^{-1}(x)))$ for $f\in (X^{(n)})_R$, $x\in (X_R)^{(n-1)}$, $n\in\en$.
\end{itemize}
Then the following hold:
\begin{itemize}
	\item[(i)] $\iota_n$ is a real-linear surjective isometry for each $n\in\en\cup\{0\}$.
	\item[(ii)] $\psi_n$ is a linear onto isometry (of real Banach spaces) for each $n\in\en\cup\{0\}$.
	\item[(iii)] $\iota_n^{-1}(\psi_n^{-1}(f))(x)=f(\psi_{n-1}(\iota_{n-1}(x)))-i f(\psi_{n-1}(\iota_{n-1}(ix)))$
	for $f\in (X_R)^{(n)}$, $x\in X^{(n-1)}$, $n\in\en$.
	\item[(iv)] For each $n\in\en\cup\{0\}$, the mappings $\iota_n$, $\psi_n$ and $\psi_n\circ\iota_n$ are weak-to-weak homeomorphisms.
	\item[(v)] For each $n\in\en$, the mapping $\psi_n\circ\iota_n$ is a weak*-to-weak* homeomorphism.
	\item[(vi)] For each $n\in\en\cup\{0\}$ we have $\psi_{n+2}\circ\iota_{n+2}\circ\kappa_{X^{(n)}}
=	\kappa_{(X_R)^{(n)}}\circ\psi_n\circ\iota_n$, where $\kappa_Y$ denotes the canonical embedding of a (real or complex) Banach space $Y$ into $Y^{**}$.
\end{itemize}
\end{prop} 
 
We continue by a transfer proposition for operators. If $X$ and $Y$ are complex Banach spaces and $T:X\to Y$ is a bounded linear operator, we denote by $T_R$ the same operator considered as an operator
from $X_R$ to $Y_R$. So, $T_R=\iota_{Y,0}\circ T\circ\iota_{X,0}^{-1}$ using the notation from Proposition~\ref{recom}. It is clear that $\|T_R\|=\|T\|$ and $(ST)_R=S_R T_R$ whenever $S:Y\to Z$ is a bounded linear operator from $Y$ to a complex Banach space $Z$.

Further, if $T:X\to Y$ is a bounded operator between two Banach spaces (real or complex, both of the same nature), we define the operators $T^{(n)}$ for $n\in\en\cup\{0\}$ inductively:
$T^{(0)}=T$ and $T^{(n)}=(T^{(n-1)})^*$ for $n\in\en$.
 
As above, we omit the straightforward proof of the following proposition.
 
\begin{prop}\label{rc-op} Let $X$ and $Y$ be complex Banach spaces and $T:X\to Y$ be a bounded linear operator.
Let $\iota_{X,n}$, $\psi_{X,n}$, $\iota_{Y,n}$ and $\psi_{Y,n}$ be the mappings from Proposition~\ref{recom} related to $X$ and $Y$, respectively. Then, for each $n\in\en\cup\{0\}$, we
have
$$\begin{aligned}
	(T_R)^{(2n)} &= \psi_{Y,2n}\circ (T^{(2n)})_R \circ \psi_{X,2n}^{-1} = 
	\psi_{Y,2n}\circ\iota_{Y,2n}\circ T^{(2n)} \circ \iota_{X,2n}^{-1}\circ \psi_{X,2n}^{-1}
	\\ (T_R)^{(2n+1)} &= \psi_{X,2n+1}\circ (T^{(2n+1)})_R \circ \psi_{Y,2n+1}^{-1}\\
	&= \psi_{X,2n+1}\circ\iota_{X,2n+1}\circ T^{(2n+1)} \circ \iota_{Y,2n+1}^{-1}\circ \psi_{Y,2n+1}^{-1}
\end{aligned}$$
\end{prop}

\subsection{Some topologies on a Banach space}
We need to define the necessary quantities. We will deal with several types of quantities -- those measuring how far is a given sequence or a net from being Cauchy, those measuring how far is a given operator from being continuous or sequentially continuous, measures of non-compactness and weak non-compactness of a set and, finally, measures of non-compactness and weak non-compactness of an operator.

We can measure non-cauchyness and non-continuity with respect to various topologies. So, we will 
give the definitions in an abstract way because we will deal with several different topologies.
Therefore we fix the following notation.

Let $X$ be a Banach space. If $F\subset X^*$ is a bounded set, let $q_F$  be the seminorm on $X$
defined by
$$q_F(x)=\sup \{ |x^*(x)| : x^*\in F\},\quad x\in X,$$
with the convention that supremum of the empty set is $0$. 

Let $\F$ be a family of subsets of the closed unit ball $B_{X^*}$ of the dual space $X^*$. Let $\tau_{\F}$ be the locally convex topology on $X$ generated by the family of seminorms $\{q_F:F\in\F\}$. In other words, $\tau_F$ is the topology of uniform convergence on the sets from $\F$.

We will work with three different families $\F$ -- the family $\F_1$  formed by all the subsets of $B_{X^*}$, $\F_2$  formed by all finite subsets of $B_{X^*}$ and $\F_3$  formed by all weakly compact subsets of $B_{X^*}$. Then $\tau_{\F_1}$ is the norm topology and $\tau_{\F_2}$ is the weak topology which we will denote by $w$. Finally, $\tau_{\F_3}$ is the restriction to $X$ of the Mackey topology on $X^{**}$ associated to the dual pair $(X^{**},X^*)$. This topology is called the \emph{Right topology} in \cite{Right,kacena}. We will denote this topology by $\rho_X$ or simply $\rho$ when $X$ is obvious.

If $X$ is a dual space, say $X=Y^*$, we define two more topologies by means of families in $B_Y$ 
(which we consider canonically embedded into $B_{Y^{**}}=B_{X^*}$). Let $\F_4$ be the family of all finite sets in $B_Y$ and $\F_5$ be the family of all weakly compact sets in $B_Y$. Then $\tau_{\F_4}$ is the weak* topology and $\tau_{\F_5}$ is the Mackey topology with respect to the dual pair $(Y^*,Y)$. We write $\rho^*_{Y^*}$ or $\rho^*$ for the topology $\tau_{\F_5}$.

In the sequel we mean by $\F$ any family of subsets of $B_{X^*}$.

The following important observation asserts that, for a complex Banach space $X$, the considered topologies coincide for $X$ and $X_R$ as well as for $X^*$ and $(X_R)^*$. Indeed, the norms in $X$ and $X_R$ are the same, the weak topology of $X$ coincides with that of $X_R$ 
(by Proposition~\ref{recom}(iv)). Further, the $\rho$-topology of $X$ coincides with that of $X_R$ as well. Indeed, let $\psi_n$ and $\iota_n$ be as in Proposition~\ref{recom}. Since $\psi_1\circ\iota_1$ is a weak-to-weak homeomorphisms, it preserves weakly compact sets. So, let $F\subset B_{X^*}$ be weakly
compact. Then obviously $q_{\psi_1(\iota_1(F))}(\iota_0x)\le q_F(x)$ for each $x\in X$ (by the very definition of $\psi_1$).
Moreover, if $F$ is absolutely convex (or at least stable by multiplying with any complex unit), then $q_{\psi_1(\iota_1(F))}(\iota_0x)= q_F(x)$ for each $x\in X$.

Since $\psi_1\circ \iota_1$ is a weak*-to-weak* homeomorphism by Proposition~\ref{recom}, weak* topology on $X^*$ coincides with the weak* topology on $(X_R)^*$. Further, similarly as for $\rho$ we obtain that $\psi_1\circ \iota_1$ is a $\rho^*$-to-$\rho^*$ homeomorphism as well.

\subsection{Quantifying non-cauchyness of sequences and nets}

Let $(x_\nu)_{\nu\in \Lambda}$ be a bounded net in $X$ indexed by a directed set $\Lambda$.
We set
\begin{equation*}
\ca[\F]{x_\nu}=\sup_{F\in\F} \inf_{\nu_0\in\Lambda} \sup\{ q_F\left(x_\nu-x_{\nu'}\right)	:
\nu,\nu'\in \Lambda, \nu\ge \nu_0, \nu'\ge \nu_0\}.
\end{equation*}
This quantity measures in a way how far the net $(x_\nu)$ is from being $\tau_{\F}$-Cauchy. 
In particular, $\ca[\F]{x_\nu}=0$ if and only if the net $(x_\nu)$ is $\tau_{\F}$-Cauchy.
It is easy to check that the quantity $\ca[\F]{\cdot}$ remains the same if we replace $\F$ by the family of all finite unions
of elements of $\F$.

The quantity $\ca[\F_2]{x_\nu}$ will be denoted by $\de{x_\nu}$. This quantity for sequences was used
already in \cite{pfitzner-inve,wesecom,qschur}. It is easy to see that $\de{x_\nu}$ is the diameter of the set of all weak$^*$ cluster points of the net $(x_\nu)$ in $X^{**}$ (we consider  $X$ canonically embedded into $X^{**}$).

The quantity $\ca[\F_1]{x_\nu}$ will be denoted simply by $\ca{x_\nu}$. This quantity for sequences was used in \cite{qschur}. The quantity $\ca[\F_3]{x_\nu}$ will be denoted by $\ca[\rho]{x_\nu}$, while the quantity $\ca[\F_5]{x_\nu^*}$ considered for a bounded net $(x_\nu^*)$ in the dual space will be denoted by $\ca[\rho^*]{x_\nu^*}$.

An important variant of these quantities is the following one. Let $(x_k)$ be a bounded sequence 
in $X$. We set
\begin{equation*}
\wca[\F]{x_k}=\inf\{ \ca[\F]{x_{k_n}} : (x_{k_n}) \mbox{ is  a subsequence of }(x_k) \}.
\end{equation*}

We will denote again the quantities $\wca[\F_1]{\cdot}$, $\wca[\F_2]{\cdot}$, $\wca[\F_3]{\cdot}$ and $\wca[\F_5]{\cdot}$
 by $\wca{\cdot}$, $\wde{\cdot}$, $\wca[\rho]{\cdot}$ and $\wca[\rho^*]{\cdot}$, respectively. Let us remark that
the quantity $\wde{\cdot}$ was used in \cite{pfitzner-inve,wesecom}.

\begin{rem} One may wonder whether the quantities $\wca[\F]{\cdot}$ should be defined using subsequences or subnets. We remark that we are using subsequences in purpose. In fact, if we defined
$\wca{\cdot}$ using subnets, we would obtain the same quantity. However, if $\wde{\cdot}$ was defined
using subnets, it would be always zero, as any bounded sequence (or even a net) in $X$ has a weakly Cauchy subnet, due to the weak* compactness of the bidual unit ball.
\end{rem} 

If $X$ is a complex Banach space, then all the quantities $\ca{\cdot}$, $\de{\cdot}$ and $\ca[\rho]{\cdot}$ are the same in $X$ and in $X_R$. For $\ca{\cdot}$ it is obvious, for $\de{\cdot}$ it is explained in \cite[Section 5]{wesecom} and it follows from Proposition~\ref{recom} using the fact that $\de{x_\nu}$ is the diameter of the weak* cluster points of $(x_\nu)$ in $X^{**}$. The equality for $\ca[\rho]{\cdot}$ follows from the easy fact that in the definition of $\ca[\rho]{\cdot}$ it is enough to take the sup over absolutely convex sets $F$ using the last paragraph of the previous subsection.

Now it is obvious that also the quantities $\wca{\cdot}$, $\wde{\cdot}$, $\wca[\rho]{\cdot}$ are the same in $X$ and in $X_R$.

Analogously we obtain that the quantity $\ca[\rho^*]{\cdot}$ remains the same in $X^*$ and $(X_R)^*$.

\subsection{Quantifying continuity and sequential continuity}

Let $X$ and $Y$ be Banach spaces. By an \emph{operator} $T:X\to Y$ we mean a bounded linear operator.
This operator is, by definition, norm-to-norm continuous. It is also weak-to-weak continuous (as $y^*\circ T$ is weakly continuous for each $y^*\in Y^*$) and $\rho$-to-$\rho$ continuous (by \cite[Lemma 12]{Right}). 

We will deal with operators which are $\rho$-to-norm continuous, $\rho$-to-norm sequentially continuous
and weak-to-norm sequentially continuous. 

Let us remark that $\rho$-to-norm continuous operators are exactly weakly compact operators. This is proved in \cite[Corollary 5]{Right}. A similar result was proved already by A.~Grothendieck. Indeed, he proved in \cite[Lemme~1]{gro} that $T$ is weakly compact if and only if $T^*$ is $\rho^*$-to-norm continuous. Note that using the Gantmacher theorem this yields one implication of the mentioned result of \cite{Right}. In Theorem~\ref{q-spojitost} below, we will prove quantitative versions of both of these results.

Weak-to-norm sequentially continuous operators are usually called \emph{completely continuous},  $\rho$-to-norm sequentially continuous operators are called \emph{pseudo weakly compact} in \cite{Right,kacena}.

For an operator $T:X\to Y$ we define the following quantities:
\begin{eqnarray*}
\cont[\F]{T}&=&\sup \{ \ca{Tx_\nu} : (x_\nu) \mbox{ is a $\tau_\F$-Cauchy net in }B_X \}, \\
\cc[\F]{T}&=&\sup \{ \ca{Tx_k} : (x_k) \mbox{ is a $\tau_\F$-Cauchy sequence in }B_X \}.
\end{eqnarray*}
Then $\cont[\F]{T}$ measures how far the mapping $T|_{B_X}$ is from being $\tau_\F$-to-norm continuous.
We will consider this quantity  for $\F_3$ and $\F_5$ and denote it as $\cont{T}$ and $\cont[\rho^*]{T}$, respectively (the quantity $\cont[\rho^*]{T}$ can be considered in case $X$ is a dual space).
It follows from \cite[Corollary 5]{Right} that $T$ is $\rho$-to-norm continuous if and only if the restriction $T|_{B_X}$ is $\rho$-to-norm continuous. Thus $\cont{T}=0$ if and only if $T$ is $\rho$-to-norm continuous (which takes place if and only if $T$ is weakly compact).	

Further, as any $\tau_\F$-Cauchy sequence is bounded, it is clear that $\cc[\F]{T}=0$ if and only if $T$ is $\tau_\F$-to-norm sequentially continuous. The quantity $\cc[\F_3]{T}$ will be denoted by $\cc[\rho]{T}$. By taking $\F=\F_2$ we get an important quantity measuring how far the operator is from being completely continuous; we denote it as $\cc{T}$, i.e.,
\[
\cc{T}=\cc[\F_2]{T}.
\] 
Similarly as above, for an operator $T$ on a dual space $X$, the quantity $\cc[\F_5]{T}$ will be denoted by $\cc[\rho^*]{T}$.

Let us remark that obviously we have
\begin{equation}
		\label{eq:cc-cont} \cc[\F]{T}\le\cont[\F]{T}
\end{equation} 
for each operator $T$.	

We finish this subsection by noticing that, if $X$ and $Y$ are complex Banach spaces
and $T:X\to Y$ is a bounded linear operator, then 
$$\cont[\rho]{T}=\cont[\rho]{T_R},\quad \cc[\rho]{T}=\cc[\rho]{T_R},\quad \cc{T}=\cc{T_R}.$$
Similarly, if $S:Y^*\to X$ is a bounded linear operator, then
\[
\cont[\rho^*]{S}=\cont[\rho^*]{S_R},\quad \cc[\rho^*]{S}=\cc[\rho^*]{S_R}.
\]
This follows immediately from the final remarks of the previous subsection.

\subsection{Measuring  non-compactness and weak non-compactness of sets}
The\-re are several ways how to measure non-compactness and weak non-compactness of a subset of a Banach space. Almost all of them need the following notation:  If $A$ and $B$ are two nonempty subsets of a Banach space $X$, we set
\begin{eqnarray*}
	\dd(A,B)&=&\inf\{\|a-b\|: a\in A, b\in B\}, \\
	\dh(A,B)&=&\sup\{\dd(a,B) : a\in A\}.
\end{eqnarray*}
Hence, $\dd(A,B)$ is the ordinary distance of the sets $A$ and $B$ and $\dh(A,B)$ is the non-symmetrized Hausdorff distance (note that the Hausdorff distance of $A$ and $B$ is equal to $\max\{\dh(A,B),\dh(B,A)\}$).

Let $A$ be a bounded subset of a Banach space $X$. Then the Hausdorff measure of non-compactness
of $A$ is defined by
\begin{equation*}
	\chi(A)=\inf\{\dh(A,F): \emptyset\ne F\subset X\mbox{ finite}\} =
	\inf\{\dh(A,K): \emptyset\ne K\subset X\mbox{ compact}\}.
\end{equation*}
The Kuratowski measure of non-compactness of $A$ is
\begin{equation*}
	\alpha(A)=\inf\{\varepsilon>0:\mbox{ there is a finite cover of $A$ by sets of diameter less than }\varepsilon\}.
\end{equation*}
We will need one more measure of non-compactness:
\begin{equation*}
 \beta(A)=\sup\{\wca{x_k} : (x_k)\mbox{ is a sequence in }A\}.
\end{equation*}
Hausdorff and Kuratowski measures of non-compactness are well known, the notation used in the literature
is not unified. It is easy to check that for any bounded set $A\subset X$ we have
\begin{equation}
		\label{eq:compact} \chi(A)\le\beta(A)\le\alpha(A)\le 2\chi(A),
\end{equation}
thus the three measures are equivalent. (And, of course, these measures equal zero if and only if the respective set is relatively compact.)

An analogue of Hausdorff measure of non-compactness for measuring weak non-compactness is
the de Blasi measure of weak non-compactness
\begin{equation*}
	\omega(A)=\inf\{\dh(A,K) : \emptyset\ne K\subset X\mbox{ is weakly compact}\}.	
\end{equation*}
Then $\omega(A)=0$ if and only if $A$ is relatively weakly compact. Indeed, the `if' part is obvious and the `only if' part follows from \cite[Lemma 1]{deblasi}.

There is another set of quantities measuring weak non-compactness. Let us name some of them:
\begin{eqnarray*}
\wk{A}&=&\dh(\wscl{A},X),\\
\wck{A}&=&\sup\{\ \dd(\clu{X}{x_k},X) : (x_k)\mbox{ is a sequence in }A\}, \\	
\gamma(A)&=&\sup\{\ |\lim_n\lim_m x^*_m(x_n)- \lim_m\lim_n x^*_m(x_n)|:      \\ && \qquad\qquad (x^*_m)\mbox{ is a sequence in } B_{X^*},  (x_n) \mbox{ is a sequence in } A  \\ && \qquad\qquad\mbox{ and all the involved limits exist}\}. 	
\end{eqnarray*}
By $\wscl{A}$ we mean the weak$^*$ closure of $A$ in $X^{**}$ (the space $X$ is canonically embedded in $X^{**}$) and $\clu{X}{x_k}$ is the set of all weak$^*$ cluster points in $X^{**}$ of  the sequence $(x_k)$. These quantities were used (explicitly or implicitly) for example in \cite{AC-meas,AC-jmaa,CKS,f-krein,Gr-krein} using different types of notation and terminology.
The quantity $\gamma$ corresponds to the Eberlein double limit criterion for weak compactness. 
It follows from \cite[Theorem 2.3]{AC-meas} that for any bounded subset $A$ of a Banach space $X$
we have
\begin{eqnarray}
\label{eq:wcomp1} &\wck{A}\le\wk{A}\le\gamma(A)\le 2\wck{A}, \\
\label{eq:wcomp2} &\wk{A}\le \omega(A).
\end{eqnarray}

So, putting together these inequalities with measures of norm non-compactness we obtain the
following diagram:
\begin{equation}
	\label{eq:diagrammsets}
	\begin{array}{cccccccccc}
	&&\chi(A)&\le&\beta(A)&\le&\alpha(A)&\le&2\chi(A) \\&&\upleq&&&&&&& \\
	&&\omega(A)&&&&&&& \\ &&\upleq&&&&&&& \\
	\wck{A}&\le&\wk{A}&\le&\gamma(A)&\le& 2\wck{A}&&
\end{array}
\end{equation}

Let us remark that the inequality $\omega(A)\le\chi(A)$ is obvious and that the quantities $\omega(\cdot)$ and $\wk{\cdot}$ are not equivalent, see \cite{tylli-cambridge,AC-meas}.
Below we show that these quantities in some spaces are equivalent. 

Their non-equivalence is illustrated also by the following theorem.

\begin{thm}\label{t:swcg} Let $X$ be a Banach space.
\begin{itemize}
	\item The space $X$ is weakly compactly generated if and only if
	$$ \forall\,\varepsilon>0\, \exists\, (A_n)_{n=1}^\infty \mbox{ a cover of }X\, \forall n\in\en: \omega(A_n)<\varepsilon.$$
	\item The space $X$ is isomorphic to a subspace of a weakly compactly generated space if and only if
	$$ \forall\,\varepsilon>0\, \exists\, (A_n)_{n=1}^\infty \mbox{ a cover of }X\, \forall n\in\en: \wk{A_n}<\varepsilon.$$
\end{itemize}
\end{thm}

Recall that $X$ is \emph{weakly compactly generated} if it admits a weakly compact subset whose linear span is dense in $X$. The first statement is an easy consequence of the fact that $X$ is weakly compactly generated if and only if it admits a norm-dense weakly $\sigma$-compact subset. The second statement is a result of \cite{f-subswcg}.

We finish this subsection again by a discussion on complex and real spaces. Let $X$ be a complex space.
Since all the measures of non-compactness $\chi(\cdot)$, $\alpha(\cdot)$ and $\beta(\cdot)$ use only the metric structure of $X$, they are the same in $X$ and in $X_R$. 

The quantity $\omega(\cdot)$ is also the same in $X$ and in $X_R$ as weak compact sets are the same and
the metric structure is the same. Further, quantities $\wk[]{\cdot}$ and $\wck[]{\cdot}$ are also the same in $X$ and in $X_R$ by Proposition~\ref{recom} (cf. also \cite[Section 5]{wesecom}). Finally, the quantity $\gamma(\cdot)$ is also the same for $X$ and for $X_R$. Indeed, let $A\subset X$ be bounded.
Let us show first that $\gamma(\iota_0(A))\le\gamma(A)$. 
Let $(x_n)$ be a sequence in $\iota_0(A)$ and $(x_m^*)$ a sequence in $B_{(X_R)^*}$ such that 
both $\lim_n \lim_m x_m^*(x_n)$ and $\lim_m \lim_n x_m^*(x_n)$ exist. Let $y_n=\iota_0^{-1}(x_n)$ 
and $y_m^*=\iota_1^{-1}(\psi_1^{-1}(x_m^*))$. By Proposition~\ref{recom}, $(y_m^*)$ is a sequence in
$B_{X^*}$ and for any $m,n\in\en$ we have
$y_m^*(y_n)=x_m^*(x_n)-ix_m^*(\iota_0(i\iota_0^{-1}(x_n)))$. Without loss of generality we can suppose that
both $\lim_n \lim_m x_m^*(\iota_0(i\iota_0^{-1}(x_n)))$ and $\lim_m \lim_n x_m^*(\iota_0(i\iota_0^{-1}(x_n)))$ exist. Then
$$|\lim_n \lim_m x_m^*(x_n)-\lim_m \lim_n x_m^*(x_n)|\le|\lim_n \lim_m y_m^*(y_n)-\lim_m \lim_n y_m^*(y_n)|\le\gamma(A).$$
By taking the supremum we get $\gamma(\iota_0(A))\le\gamma(A)$.

Conversely, suppose $\gamma(A)>c$. Fix $(x_n)$  a sequence in $A$ and $(x_m^*)$ a sequence in $B_{X^*}$ such that 
$$|\lim_n \lim_m x_m^*(x_n)-\lim_m \lim_n x_m^*(x_n)|>c$$
and all the limits involved exist. Let $\alpha$ be a complex unit such that
$$|\lim_n \lim_m x_m^*(x_n)-\lim_m \lim_n x_m^*(x_n)|=\alpha(\lim_n \lim_m x_m^*(x_n)-\lim_m \lim_n x_m^*(x_n)).$$
Then
\begin{multline*}
\lim_n \lim_m \psi_1(\iota_1(\alpha x_m^*))(\iota_0 x_n)-\lim_m \lim_n \psi_1(\iota_1(\alpha x_m^*))(\iota_0 x_n) \\
=|\lim_n \lim_m x_m^*(x_n)-\lim_m \lim_n x_m^*(x_n)|>c,
\end{multline*}
hence $\gamma(\iota_0(A))>c$, which gives $\gamma(A)\le\gamma(\iota_0(A))$.

\subsection{Measuring non-compactness and weak non-compactness of operators}

An operator $T:X\to Y$ is compact (weakly compact) if $T(B_X)$ is a relatively compact
(relatively weakly compact, respectively) subset of $Y$. Therefore, if we want to measure how far a given operator is from being compact (weakly compact), we can use one of the measures of non-compactness
(weak non-compactness) defined in the previous section. To simplify the notation we adopt the following convention. By a quantity applied to $T$ we mean this quantity applied to $T(B_X)$. So, in particular,
$\chi(T)$, $\omega(T)$ and $\wk[Y]{T}$ denote $\chi(T(B_X))$, $\omega(T(B_X))$ and $\wk[Y]{T(B_X)}$, respectively. Due to the previous subsection, these quantities are the same for $T$ and $T_R$ in case $X$ and $Y$ are complex spaces. 
	
Another possibility is to measure the distance to compact (weakly compact) operators. The distance of $T$ to the space of compact operators is denoted by $\|T\|_K$ and is called the \emph{essential norm}
of $T$. The distance to the space of weakly compact operators is denoted by $\|T\|_w$ and is called
\emph{weak essential norm}.

By the Schauder theorem, $T$ is compact if and only if $T^*$ is compact. Similarly, the Gantmacher theorem says that $T$ is weakly compact if and only if $T^*$ is weakly compact. Both theorems have quantitative versions, as for any operator $T$ we have
\begin{eqnarray}
		\label{eq:Schauder} \frac12\chi(T)\le\chi(T^*)\le2\chi(T), \\
		\label{eq:Gant-gamma} \gamma(T)\le\gamma(T^*)\le 2\gamma(T).
\end{eqnarray}
The inequality \eqref{eq:Schauder} is a result of \cite{GM}, the inequality \eqref{eq:Gant-gamma}
is proved in \cite[Theorem~3.1]{AC-meas}. By combining \eqref{eq:Gant-gamma} with \eqref{eq:wcomp1}
we get
\begin{equation}
		\label{eq:Gant-dh} \frac12 \wk[Y]{T}\le \wk[X^*]{T^*} \le 2\wk[Y]{T}.
\end{equation}
These results were originally proved for real spaces. However, they hold for complex spaces as well, due to the fact that the quantities $\chi(\cdot)$, $\gamma(\cdot)$ and $\wk[]{\cdot}$ are the same for $T^*$ and for $(T_R)^*$.
Indeed, using Proposition~\ref{rc-op} and Proposition~\ref{recom} we get
$$(T_R)^*(B_{(Y_R)^*})=\psi_{X,1}(\iota_{X,1}(T^*(\iota_{Y,1}^{-1}(\psi_{Y,1}^{-1}(B_{(Y_R)^*})))))
=\psi_{X,1}(\iota_{X,1}(T^*(B_{Y^*}))).$$
So, using again Proposition~\ref{recom}, we see that the quantities $\chi(\cdot)$ and $\wk[]{\cdot}$ (and also $\alpha(\cdot)$, $\beta(\cdot)$ and $\wck[]{\cdot}$) are the same for $T^*$ and $(T_R)^*$. Further, the quantity $\gamma(\cdot)$ is also the same, as by the previous section
$\gamma(\iota_{X,1}(T^*(B_{Y^*})))=\gamma(T^*(B_{Y^*}))$ and $\psi_{X,1}$ is just a linear isometry of real spaces.  	

We have thus the following diagrams:
\begin{equation}
		\label{eq:diagramm-operators}
		\begin{array}{ccc}
		\chi(T)&\le&\|T\|_K \\ \upsim&&\upleq \\ \chi(T^*)&\le&\|T^*\|_K
		\end{array} \qquad
		\begin{array}{ccccc}
		\wk[Y]{T}&\le&\omega(T)&\le& \|T\|_w\\
		\upsim & & & & \upleq \\
		\wk[X^*]{T^*}&\le&\omega(T^*)&\le&\|T^*\|_w
		\end{array}
\end{equation}

The exact meaning of the equivalence sign is given by \eqref{eq:Schauder} and \eqref{eq:Gant-dh}.
The other inequalities are either trivial or a consequence of the Schauder and the Gantmacher theorem.
In general, there are no other inequalities (even including a multiplicative constant). For the first diagram it follows from \cite{tylli-jfa,tylli-israel}, for the second one from \cite{tylli-cambridge,tylli-pams}. In particular, the quantities $\omega(T)$ and $\omega(T^*)$ are in general incomparable.	

%%%%%%%%%%%%%%%%%%%%%%%%%%%%%%%%%%%%%%
\section{Easy quantitative implications}\label{sec-easy}

Any compact operator is obviously weakly compact. Further, any compact operator is easily seen to be completely continuous. It is also easy to see that any operator which is either weakly compact or completely continuous maps weakly Cauchy sequences to weakly convergent sequences. Such operators are
called \emph{weakly completely continuous}. We have thus the following implications:
\begin{equation*}
\begin{array}{ccc}
	T\mbox{ is compact} &\Rightarrow & T\mbox{ is completely continuous} \\
	\Downarrow &&\Downarrow \\
		T\mbox{ is weakly compact} &\Rightarrow & T\mbox{ is weakly completely continuous} 
\end{array}	
\end{equation*}

These implications have quantitative versions. We have already defined quantities measuring how far a given operator is from being compact, weakly compact or completely continuous. To formulate all the inequalities, we need to define, for a given operator $T:X\to Y$, the following two quantities:
$$\begin{aligned}
\wcc(T) & =  \sup \{ \dh(\clu{Y}{Tx_k},Y) : (x_k)\mbox{ is a weakly Cauchy sequence in }B_X \} \\
&=	\sup \{ \wk[Y]{\{Tx_k:k\in\en\}} : (x_k)\mbox{ is a weakly Cauchy sequence in }B_X \},\\
\wcc_\omega(T)&= 	\sup \{ \omega(\{Tx_k:k\in\en\}) : (x_k)\mbox{ is a weakly Cauchy sequence in }B_X \}.
\end{aligned}$$

The promised quantitative versions of the above implications are contained in the following table:

\begin{equation}
		\label{eq:q-easyimp}
		\begin{array}{ccccc}
\cc{T}&\sleq&	\chi(T)&\le&\|T\|_K \\
\upleq&&\upleq&&\upleq\\
\wcc_{\omega}(T)&\le&\omega(T)&\le&\|T\|_w\\
\upleq&&\upleq&&\\
\wcc(T)&\le&\wk[Y]{T}&&
\end{array}
\end{equation}

The sign $\sleq$ means that the inequality holds with a universal positive multiplicative constant which in this case is $4$ by~\eqref{eq:ccchi}.

Most of the inequalities included in the diagram are easy and are immediate consequence of the inequalities \eqref{eq:diagrammsets} and \eqref{eq:diagramm-operators}. We will prove the remaining two inequalities, i.e.,
\begin{eqnarray}
		\label{eq:ccchi} \cc{T}&\le&4\chi(T), \\ 
		\label{eq:wccc} \wcc_\omega(T)&\le&\cc{T}.
\end{eqnarray}

To prove the first one we need the following lemma.

\begin{lemma}
Let $X$ be a Banach space and $(x_k)$ be a weakly Cauchy sequence in~$X$. Let $c>0$ be such that $\ca{x_k}>c$. Then there is a subsequence $(x_{k_n})$ such that $\wca{x_{k_n}}\ge\frac c2$.
\end{lemma}

\begin{proof}
If $(x_k)$ is weakly Cauchy, it weak$^*$ converges to some $x^{**}\in X^{**}$. If $\ca{x_k}>c$, then
$$\forall n\in\en\,\exists k,l\ge n:\|x_k-x_l\|>c.$$
By the triangle inequality we get that
$$\forall n\in\en\,\exists k\ge n:\|x_k-x^{**}\|>\frac c2.$$
It follows that there is a subsequence $(x_{k_n})$ such that $\|x_{k_n}-x^{**}\|>\frac c2$ for each $n\in\en$. We claim that $\wca{x_{k_n}}\ge\frac c2$.

Indeed, let $(y_l)$ be any subsequence of $(x_{k_n})$ and $m\in\en$ be arbitrary. Then
$$\diam\{y_l: l\ge m\}=\diam\wscl{\{y_l: l\ge m\}}=\diam(\{y_l: l\ge m\}\cup\{x^{**}\})>\frac c2,$$
hence $\ca{y_l}\ge \frac c2$. This completes the proof.
\end{proof}

Now we are going to prove inequality \eqref{eq:ccchi}. Due to \eqref{eq:compact} it is enough to prove
\begin{equation}
		\label{eq:ccbeta} \cc{T}\le2\beta(T).
\end{equation}
If $\cc{T}=0$, the inequality is obvious. Suppose that $\cc{T}>0$ and fix any $c>0$ satisfying $\cc{T}>c$. Then there is a weakly Cauchy sequence $(x_k)$ in $B_X$ with $\ca{Tx_k}>c$. Since $(Tx_k)$ is 
weakly Cauchy as well, the above lemma yields a subsequence $(x_{k_n})$ with $\wca{Tx_{k_n}}\ge\frac c2$.
By the definition of $\beta$ we get $\beta(T)=\beta(T(B_X))\ge\frac c2$. Since $c<\cc{T}$ is arbitrary,
we get $\beta(T)\ge\frac12\cc{T}$ which yields \eqref{eq:ccbeta}.

We proceed to the proof of \eqref{eq:wccc}. If $\wcc_\omega(T)=0$, the inequality is obvious.
Suppose now that $\wcc_\omega(T)>c>0$. Then there is a weakly Cauchy sequence $(x_k)$ in $B_X$ with
$\omega(\{Tx_k:k\in\en\})>c$. Since, for each $n\in\en$, we have $\omega(\{Tx_k:k\ge n\})>c$ and the singleton $\{Tx_n\}$ is weakly compact,
$\diam\{Tx_k:k\ge n\}>c$.  Thus $\ca{Tx_k}\ge c$. Since $c<\wcc_\omega(T)$ is arbitrary, we get $\cc{T}\ge\wcc_\omega(T)$, and so the proof of \eqref{eq:wccc} is complete.

%%%%%%%%%%%%%%%%%%%%%%%%%%%%%%%%%%%%%%%%%%%%%%%%%%%%%%%%%%%%
\section{Weak compactness of operators and continuity}
\label{sec-wco}

The first of our main results are quantitative versions of \cite[Corollary~5]{Right} and \cite[Lemme~1]{gro}. This section is devoted to their proofs.

\begin{thm}\label{q-spojitost} Let $X$ and $Y$ be Banach spaces and $T:X\to Y$ be a bounded linear operator.
Then
\begin{eqnarray}
\label{eq:rho-T*}\tfrac12\cont[\rho]{T}\le\omega(T^*)\le \cont[\rho]{T},\\
\label{eq:rho*-T}\tfrac12\cont[\rho^*]{T^*}\le \omega(T)\le \cont[\rho^*]{T^*}.
\end{eqnarray}
\end{thm}

The first assertion \eqref{eq:rho-T*} is  the promised quantitative version of \cite[Corollary 5]{Right}.
We stress that the $\rho$-to-norm continuity of $T$ is quantitatively equivalent to the weak compactness
of $T^*$, not to that of $T$. (Recall that $\omega(T^*)$ is not equivalent to $\omega(T)$.) 
The second chain of inequalities \eqref{eq:rho*-T} is a quantitative variant of A.~Grothendieck's result stating that an operator $T$ is weakly compact if and only if $T^*$ is $\rho^*$-to-norm continuous (see \cite[Lemme~1]{gro}).

\begin{proof}[Proof of Theorem~\ref{q-spojitost}]
Let $X$ and $Y$ be Banach spaces and $T:X\to Y$ be a bounded linear operator. We start the proof with the inequality
\begin{equation}
\label{q-s1}
\frac12\cont[\rho^*]{T^*}\le\omega(T).
\end{equation}
Let $c>\omega(T)$ and $(y_\nu^*)_{\nu \in \Lambda}$ be an arbitrary $\rho^*$-Cauchy net in $B_{Y^*}$. 
We will show that $\ca{T^*y_\nu^*}\le 2c$.

By the definition, there exists a nonempty weakly compact set $L\subset Y$ such that 
\[
T(B_X)\subset L+cB_{Y}.
\]
Since $L$ is weakly compact, the net $(y_\nu^*)$ is uniformly Cauchy on $L$ (note that $L$ is bounded,
hence a positive multiple of $L$ is contained in $B_Y$).

Let $\ep>0$ be arbitrary. There exists $\nu_0\in \Lambda$ such that
\begin{equation}
\label{q-s2}
\sup_{y\in L} |y_{\nu}^*(y) -y_{\nu'}^*(y)|<\ep,\quad \nu,\nu'\geq \nu_0.
\end{equation}
Given $x\in B_{X}$, let $y\in L$ satisfy $\|Tx-y\|\le c$. By \eqref{q-s2},
\[
|(y_\nu^*-y_{\nu'}^*)(y)|<\ep,\quad \nu,\nu'\geq \nu_0.
\]
Thus  we have for $\nu,\nu'\geq \nu_0$
\[
\aligned
|(T^*y_{\nu}^*-T^*y_{\nu'}^*)(x)|&=|(y_{\nu}^*-y_{\nu'}^*)(Tx)|\\
&\le |(y_{\nu}^*-y_{\nu'}^*)(Tx-y)|+|(y_{\nu}^*-y_{\nu'}^*)(y)|\\
&\le 2c+\ep.
\endaligned
\]
Thus we get for $\nu,\nu'\geq \nu_0$
\[
\|T^*y_{\nu}^*-T^*y_{\nu'}^*\|=\sup_{x\in B_X} |(T^*y_{\nu}^*-T^*y_{\nu'}^*)(x)|\le 2c+\ep.
\]
It follows that $\ca{T^*y_\nu^*}\le 2c+\ep$. Since $\ep>0$ is arbitrary, we get $\ca{T^*y_\nu^*}\le 2c$. Hence $\cont[\rho*]{T^*}\le 2c$, which yields \eqref{q-s1}.

Next we observe that
\begin{equation}
\label{eq:T-T*}
\cont[\rho]{T}\le \cont[\rho^*]{T^{**}},
\end{equation}
since any $\rho$-Cauchy net $(x_\nu)$ in $B_X$ is $\rho^*$-Cauchy in $B_{X^{**}}$.
Using \eqref{eq:T-T*} and \eqref{q-s1} for $T^*$ we get
\[
\frac12\cont[\rho]{T}\le \frac12\cont[\rho^*]{T^{**}}\le \omega(T^*),
\]
which proves the first half of \eqref{eq:rho-T*}.

It remains to verify the second inequalities in \eqref{eq:rho-T*} and \eqref{eq:rho*-T}.
In order to prove
\begin{equation}
\label{q-s3}
\omega(T^*)\le \cont[\rho]{T},
\end{equation}
let us fix an arbitrary $c>\cont[\rho]{T}$.  We claim that:

\emph{There exists a $\rho$-neighborhood $U$ of $0$ such that $\|Tx\|\le c$ for every $x\in U\cap B_X$.}

Assuming the contrary, we can find for  every $\rho$-neighborhood $U$ of $0$ an element $x_U\in U\cap B_X$ such that $\|Tx_{U}\|> c$.
Let $\U$ denote the family of all  $\rho$-neighborhoods of $0$. We consider $\U$ endowed with the partial order given by inverse inclusion, and thus $(x_U)_{U\in \U}$ is a net converging to $0$ in the topology $\rho$. We further consider a directed set $\U\times\{0,1\}$ with the lexicographical ordering and set
\[
x_{U,i}=\begin{cases}
             x_{U},& i=0,\\
             0,& i=1,
        \end{cases}
        \ U\in\U.
\]
Then $(x_{U,i})$ is again a net in $B_X$ converging to $0$ in the topology $\rho$, and thus $\ca{Tx_{U,i}}\le\cont[\rho]{T}<c$. On the other hand,
\[
\|Tx_{U,1}-Tx_{U,0}\|=\|Tx_U\|>c
\]
for any $U\in\U$, which is a contradiction completing the proof of the claim.

\smallskip

Let $U$ be the $\rho$-neighborhood of $0$ from the claim.
By the definition of $\rho$, there exist $d>0$ and weakly compact sets $K_1,\dots, K_n$ in $B_{X^*}$ such that
\[
U\supset\{x\in X\colon \sup_{x^*\in K_i} |x^*(x)|<d,\ i=1,\dots,n\}.
\]
By the Krein theorem, the closed absolutely convex hull $K$ of $K_1\cup\cdots\cup K_n$ is a weakly compact subset of $B_X$, and thus we may assume that
\[
U=\{x\in X\colon \sup_{x^*\in K} |x^*(x)|<d\}.
\]

To find a weakly compact set needed by the definition of $\omega(T^*)$, we use the following assertion:

\emph{There exists $n\in\en$ such that $T^*(B_{Y^*})\subset nK+cB_{X^*}$.}

To verify this, assume that this is not the case. Then for every $n\in\en$ there exists $y_n^*\in B_{Y^*}$ with
\[
T^*y_n^*\notin nK+cB_{X^*}.
\]
The  set $nK$ is weakly compact, hence also weak* compact. It follows that $nK+cB_{X^*}$ is a weak* compact absolutely convex set, and thus we may separate the point $T^*y_n^*$ from it by an element $x_n\in X$ of norm one such that
\[
\sup_{x^*\in nK, z^*\in B_{X^*}} \Re (x^*(x_n)+cz^*(x_n))<\Re(T^*y_n^*)(x_n).
\]
Since $K$ is absolutely convex, we get
$$\sup_{x^*\in nK} |x^*(x_n)|+c<\Re(T^*y_n^*)(x_n).$$

Let $n\in\en$ be such that $\frac{\|T^*\|}{n}<d$. Then 
\[
|(nx^*)(x_n)|< |nx^*(x_n)|+c<\Re(T^*y_n^*)(x_n)\leq \|T^*\|,\quad x^*\in K,
\]
and thus
\[
|x^*(x_n)|<d,\quad x^*\in K.
\]
Hence $x_n\in U\cap B_X$, which implies $\|Tx_n\|\le c$ by the choice of $U$. Thus
\[
\aligned
c&\ge\|Tx_n\|\ge\Re y_n^*(Tx_n)=\Re(T^*y_n^*)(x_n)\\
&> \sup_{x^*\in nK} |x^*(x_n)|+c\\
&\ge c.
\endaligned
\]

The contradiction proves the assertion, and so we have $\omega(T^*)\leq c$. This finishes the proof of \eqref{q-s3}, and thus also of \eqref{eq:rho-T*}.

Since the proof of the remaining inequality
\begin{equation}
\label{eq:last-rho}
\omega(T)\le \cont[\rho^*]{T^*}
\end{equation}
is rather analogous to the one of \eqref{q-s3}, we merely outline it. Given $c>\cont[\rho^*]{T^*}$, we find a $\rho^*$-neighbourhood $V$ of $0$ in $Y^*$ such that $\|T^*y^*\|\le c$ for every $y^*\in V\cap B_{Y^*}$. Then we may assume without loss of generality that there are an absolutely convex weakly compact set $L\subset B_Y$ and $d>0$ such that $V=\{y^*\in Y^*: \sup_{y\in L} |y^*(y)|<d\}$. Finally, we find $n\in\en$ such that $T(B_X)\subset nL+cB_Y$. (We proceed as above, assuming the contrary, for each $n\in\en$ there exist $x_n\in B_X$ and $y_n^*\in B_{Y^*}$ satisfying
\[
\sup_{y\in nL} |y_n^*(y)|+c<\Re y_n^*(Tx_n).
\]
For $n\in\en$ with $\frac{\|T\|}{n}<d$, we then get $y_n^*\in V$, and thus
\[
c\ge \|T^*y_n^*\|\ge\Re y_n^*(Tx_n)>c,
\]
which is a contradiction.)
\end{proof}

%%%%%%%%%%%%%%%%%%%%%%%%%%%%%%%%%%%%%%%%%%%%%%
\section{Two ways of quantifying the Dunford-Pettis property}
\label{sec-twoways}

We recall that a Banach space $X$ is said to have the \emph{Dunford-Pettis property} if for any Banach space $Y$ every weakly compact operator $T:X\to Y$ is completely continuous. The following theorem summarizes the well-known equivalent formulations of this property.

\begin{thm}
\label{t:znama}
For a Banach space $X$, the following assertions are equivalent:
\begin{itemize}
\item [(i)] $X$ has the Dunford-Pettis property,
\item [(ii)] every weakly compact operator $T:X\to c_0$ is completely continuous,
\item [(iii)] given a weakly null sequence $(x_n)$ in $X$ and a weakly null sequence $(x_n^*)$ in $X^*$, we have $\lim_n x_n^*(x_n)=0$,
\item [(iv)] weakly convergent sequences in $X$ coincide with $\rho$-convergent ones,
\item [(v)] every weakly convergent sequence in $X^*$ is $\rho^*$-convergent,
\item [(vi)] if $T: Y\to X$ is weakly compact, with $Y$ an arbitrary Banach space, then $T^*$ is completely continuous,
\item [(vii)] if $T: \ell_1\to X$ is weakly compact, then $T^*$ is completely continuous.
\end{itemize}
\end{thm}
\begin{proof}
The proofs of many of the equivalences involved in Theorem~\ref{t:znama} are almost identical and use the same techniques for which we refer to \cite{diestel} or \cite{kacena}. The equivalence of (i), (ii), (iii) and (vi) is mentioned in \cite[Theorem~1]{diestel}, the equivalence of (i), (iv) and (v) has been basically proved in the context of locally convex spaces by A.~Grothendieck in \cite[Proposition~1~bis]{gro}. Note that for the implication (v) $\Rightarrow$ (vi) one needs only the aforementioned result that for any weakly compact operator~$T$, the adjoint $T^*$ is $\rho^*$-to-norm continuous. The implication (vi) $\Rightarrow$ (vii) is trivial and for (vii) $\Rightarrow$ (iii) it is enough to consider the operator $T:\ell_1\to X$ with $T(a_n)=\sum a_n x_n,\,(a_n)\in\ell_1,$ where $(x_n)$ is a given weakly null sequence in~$X$.
\end{proof}

Using the results of Section~\ref{sec-wco} we obtain that the Dunford-Pettis property is always quantitative in some sense.

\begin{thm}\label{q-DP-univ} For a Banach space $X$, the following assertions are equivalent:
\begin{itemize}
\item [(i)] $X$ has the Dunford-Pettis property,
\item [(ii)] $\cc{T}\le 2\omega(T^*)$ for every operator $T$ from $X$ to any Banach space $Y$,
\item [(iii)] $\cc{T^*}\le 2\omega(T)$ for every operator $T$ from any Banach space $Y$ to $X$,
\item [(iv)] $\limsup|x_n^\ast(x_n)|\le\omega({\{x_n^*:n\in\en\}})$ whenever $(x_n)$ is a weakly null sequence in~$B_{X}$ and $(x_n^*)$ is a bounded sequence in~$X^*$,
\item [(v)] $\limsup|x_n^\ast(x_n)|\le\omega({\{x_n:n\in\en\}})$ whenever $(x_n)$ is a bounded sequence in~$X$ and $(x_n^\ast)$ is a weakly null sequence in~$B_{X^\ast}$,
\item [(vi)] $\wca[\rho^*]{x_n^*}\le2\omega(\{x_n^*:n\in\en\})$ whenever $(x_n^*)$ is a bounded sequence in~$X^*$,
\item [(vii)] $\wca[\rho]{x_n}\le2\omega(\{x_n:n\in\en\})$ whenever $(x_n)$ is a bounded sequence in~$X$.
\end{itemize}
\end{thm}
\begin{proof}
Obviously, by Theorem~\ref{t:znama}, any of the assertions (ii)--(v) implies assertion (i). For (vi) and (vii) this follows from the completeness of the Mackey topologies $\rho^*=\tau(X^*,X)$ and $\tau(X^{**},X^*)$, respectively (see \cite[Proposition~1.1]{sch-whe}). Indeed, suppose for example (vii). Let $(x_n)$ be a weakly null sequence. Then any subsequence $(x_{k_n})$ is also weakly null and hence (by (vii)) $\wca[\rho]{x_{k_n}}=0$. It follows that any subsequence of $(x_n)$ has a further subsequence which is $\rho$-Cauchy. Indeed, let $(u_n)$ be any subsequence of $(x_n)$. Set $u_n^0=u_n$ and construct by induction $(u_n^k)$ a subsequence of $(u_n^{k-1})$ with $\ca[\rho]{u_n^k}<\frac1k$. The diagonal sequence $(u_k^k)$ is then $\rho$-Cauchy. 
By the aforementioned completeness it follows that any weakly null $\rho$-Cauchy sequence is $\rho$-null. Thus any subsequence of $(x_n)$ has a further subsequence which is $\rho$-null. Therefore $(x_n)$ itself is $\rho$-null. We have 
proved that $X$ satisfies the condition (iv) of Theorem~\ref{t:znama}, hence $X$ has the Dunford-Pettis property.
The reasoning for (vi) $\Rightarrow$ (i) is similar. Thus it is sufficient to show that the Dunford-Pettis property implies all the other assertions.

(i) $\Rightarrow$ (ii)
Suppose $X$ has the Dunford-Pettis property, $Y$ is any Banach space and $T:X\to Y$ is a bounded linear operator. By Theorem~\ref{t:znama}(iv), weakly Cauchy and $\rho$-Cauchy sequences in $X$ coincide, thus in particular $\cc{T}=\cc[\rho]{T}$. Since obviously $\cc[\rho]{T}\le\cont[\rho]{T}$ (cf. \eqref{eq:cc-cont}), Theorem~\ref{q-spojitost} gives (ii).

(i) $\Rightarrow$ (iii)
Similarly, suppose $X$ has the Dunford-Pettis property, $Y$ is any Banach space and $T:Y\to X$ is a bounded linear operator. By Theorem~\ref{t:znama}(v), weakly Cauchy sequences in $X^*$ are $\rho^*$-Cauchy, thus in particular $\cc{T^*}\le\cc[\rho^*]{T^*}$. Since obviously $\cc[\rho^*]{T}\le\cont[\rho^*]{T^*}$, Theorem~\ref{q-spojitost} yields (iii).

(i) $\Rightarrow$ (v)
Let $(x_n)$ be a bounded sequence in~$X$ and $(x_n^\ast)$ be a weakly null sequence in~$B_{X^\ast}$. 
Let $c>\omega(\{x_n:n\in\en\})$ be arbitrary. Fix a weakly compact set $K\subset X$ such that $\dh(\{x_n:n\in\en\},K)<c$. For each $n\in\en$, let $y_n\in K$ be such that $\|y_n-x_n\|<c$.
Since $(x_n^\ast)$ is weakly null, it is also $\rho^*$-null (by Theorem~\ref{t:znama}), so in particular
$x_n^\ast\to 0$ uniformly on $K$. It follows that $x_n^*(y_n)\to 0$.
Hence
$$\limsup |x_n^*(x_n)| \le 
\limsup |x_n^*(x_n-y_n)| + \limsup |x_n^*(y_n)| \le \limsup\|x_n-y_n\|\le c.$$ 
This completes the proof.

(i) $\Rightarrow$ (iv)
This implication can be proved exactly as the previous one, we only need to interchange roles of $X$ and $X^*$.

(i) $\Rightarrow$ (vii)
Let $c>\omega(\{x_n:n\in\en\})$ be arbitrary. Fix a weakly compact set $K\subset X$ such that $\dh(\{x_n:n\in\en\},K)<c$. For each $n\in\en$, let $y_n\in K$ be such that $\|y_n-x_n\|<c$.
Since $K$ is weakly compact, there is a subsequence $(y_{n_k})$ weakly converging to some $y\in K$.
Then $(y_{n_k})$ is also $\rho$-convergent (by Theorem~\ref{t:znama}). To complete the proof it is enough to show that $\ca[\rho]{x_{n_k}}\le 2c$.
Fix any weakly compact $L\subset B_{X^*}$. Then for any $k,l\in\en$ we have
$$q_L(x_{n_k}-x_{n_l})\le q_L(x_{n_k}-y_{n_k})+q_L(y_{n_k}-y_{n_l})+q_L(y_{n_l}-x_{n_l})< 2c+q_L(y_{n_k}-y_{n_l}).$$
It follows that $$\ca[\rho]{x_{n_k}}\le 2c+\ca[\rho]{y_{n_k}}=2c$$ and the proof is completed.
 
(i) $\Rightarrow$ (vi)
This implication can be proved analogously to the previous one by interchanging roles of $X$ and $X^*$.
\end{proof}

\begin{remark}
Quantities $\wca[\rho^*]{\cdot}$ and $\wca[\rho]{\cdot}$ in the assertions (vi) and (vii), respectively, of Theorem~\ref{q-DP-univ} cannot be replaced by $\ca[\rho^*]{\cdot}$ and $\ca[\rho]{\cdot}$. Indeed, let $X$ be an arbitrary Banach space possessing the Dunford-Pettis property. Consider a sequence  $(x_n)$ of the form $x,0,x,0,\ldots$ in $X^*$ (in $X$) with $x\neq0$. Then $\omega(\{x_n:n\in\en\})=0$, but $\ca[\rho^*]{x_n}=\|x\|$ ($\ca[\rho]{x_n}=\|x\|$, respectively).
\end{remark}

It is natural to ask whether a variant of Theorem~\ref{q-DP-univ} can be proved with quantities $\omega(\cdot)$ replaced by the respective quantities $\wk[]{\cdot}$. Interestingly enough, the study of this question brings us to deeper understanding of the Dunford-Pettis property. It turns out that the analogues of conditions (ii), (iv) and (vi) with $\omega(\cdot)$ replaced by $\wk[]{\cdot}$ are all equivalent to each other and so are the analogues of conditions (iii), (v) and (vii). Both groups of these quantitative assertions obviously strengthen the Dunford-Pettis property, however, as Example~\ref{exa-podrobny} will show, they are incomparable in general. This reveals the dual nature of the Dunford-Pettis property which is not apparent in the classical non-quantitative case.

\begin{thm}\label{q-DP-direct}
Let $X$ be a Banach space. The following assertions are equivalent:
\begin{enumerate}
\item[(i)]
There is $C>0$ such that $\cc{T}\le C\wk[X^*]{T^*}$ for any operator $T$ from $X$ to a Banach space $Y$.
\item[(ii)]
There is $C>0$ such that $\cc{T}\le C\wk[X^*]{T^*}$ for any operator $T$ from $X$ to $\ell_\infty$.
\item[(iii)]
There is $C>0$ such that $\limsup|x_n^\ast(x_n)|\le C\wk[X^*]{\{x_n^*:n\in\en\}}$ whenever $(x_n)$ is a weakly null sequence in~$B_X$ and $(x_n^*)$ is a bounded sequence in~$X^*$.
\item[(iv)]
There is $C>0$ such that $\ca[\rho^*]{x_n^*}\le C\de{x_n^*}$ for any bounded sequence $(x_n^*)$ in~$X^*$.
\item[(v)]
There is $C>0$ such that $\wca[\rho^*]{x_n^*}\le C\wk[X^*]{\{x_n^*:n\in\en\}}$ for any bounded sequence $(x_n^*)$ in~$X^*$.
\item[(vi)]
There is $C>0$ such that $\cc{T}\le C\wk[Y]{T}$ for any operator $T$ from $X$ to a Banach space $Y$.
\item[(vii)]
There is $C>0$ such that $\cc{T}\le C\wk[\ell_\infty]{T}$ for any operator $T$ from $X$ to $\ell_\infty$.
\end{enumerate}
\end{thm}

\begin{proof}
The implication (i) $\Rightarrow$ (ii) holds trivially, even with the same constant.

(ii) $\Rightarrow$ (iii)
Let us assume that there is $C>0$ such that $\cc{T}\le C\wk[X^*]{T^*}$ for any operator $T$ from $X$ to $\ell_\infty$. Let $(x_n)$ be a a weakly null sequence in~$B_X$ and $(x_n^*)$ be a bounded sequence in~$X^*$. We will show that
\[\limsup|x_n^\ast(x_n)|\le 8C\wk[X^*]{\{x_n^*:n\in\en\}}.\]

Let us define operator $S:\ell_1\to X^*$ by $S(\lambda_n)=\sum_n\lambda_n x_n^*$. Since $Se_n=x_n^*$ for every $n\in\en$, where $e_n$ denotes the $n$-th basic vector in $\ell_1$, the set $S(B_{\ell_1})$ is contained in the closed absolutely convex hull of $\{x_n^*:n\in\en\}$, and so, by \cite[Theorem~2]{f-krein},
\begin{equation}
\label{eq:q-DP-dual1}
\wk[X^*]{S}\le 2\wk[X^*]{\{x_n^*:n\in\en\}}.
\end{equation}
In fact, the result of \cite{f-krein} is formulated for the closed convex hull, but the result on the closed absolutely convex hull is an easy consequence (both in the real and  the complex cases).

Let $T$ be the restriction of $S^*$ to the space $X$. Then $T$ is an operator from $X$ to $\ell_\infty$. Using the fact that $(x_1,0,x_2,0,\ldots)$ is a weakly Cauchy sequence in $B_X$, the assumption (ii) and estimates \eqref{eq:Gant-dh} and \eqref{eq:q-DP-dual1}, we can write
\[
\begin{aligned}
\limsup|x_n^\ast(x_n)| &= \limsup|e_n(Tx_n)| \le \limsup \|Tx_n\|\\
&\le \cc{T} \le C\wk[X^*]{T^*}\\
&\le 2C\wk[\ell_\infty]{T} \le 2C\wk[\ell_\infty]{S^*}\\
&\le 4C\wk[X^*]{S} \le 8C\wk[X^*]{\{x_n^*:n\in\en\}}.
\end{aligned}
\]

(iii) $\Rightarrow$ (iv) Let us assume that (iii) holds with a constant $C>0$. We will show that (iv) holds with the constant $2C+1$.  Let $(x_n^*)$ be a bounded sequence in~$X^*$. If $\ca[\rho^*]{x_n^*}=0$, the inequality is obvious. So, suppose $\ca[\rho^*]{x_n^*}>0$ and fix any
$t\in(0,\ca[\rho^*]{x_n^*})$. Then there is a sequence of natural numbers $l_n<m_n<l_{n+1},\,n\in\en,$ and a weakly compact set $K\subset B_X$ such that $q_K(x_{l_n}^*-x_{m_n}^*)>t$ for every $n\in\en$. Let $(x_n)$ be a sequence in~$K$ such that $|(x_{l_n}^*-x_{m_n}^*)(x_n)|>t$ for every $n\in\en$. By passing to a subsequence if necessary, we may assume that $(x_n)$ is weakly convergent to some $x\in K$. Then the sequence $(y_n)=(\frac{x_n-x}2)$ is a weakly null sequence in~$B_X$. 

Any weak$^*$ cluster point of the sequence $(x_{l_n}^*-x_{m_n}^*)$ in $X^{***}$ is the difference of two weak$^*$ cluster points of $(x_n^*)$ in $X^{***}$, in particular
\begin{equation}
		\label{eq:wkde} \wk[X^*]{x_{l_n}^*-x_{m_n}^*}\le\de{x_n^*}.
\end{equation} 
Using consecutively the fact that $x_n=2y_n+x$, the validity of (iii) with $C$ and \eqref{eq:wkde}, we get
$$
\begin{aligned}
t&\le\liminf |(x_{l_n}^*-x_{m_n}^*)(x_n)|\\
&\le 2\limsup |(x_{l_n}^*-x_{m_n}^*)(y_n)| + \limsup |(x_{l_n}^*-x_{m_n}^*)(x)|\\
&\le 2C\wk[X^*]{\{x_{l_n}^*-x_{m_n}^*: n\in\en\}}+\wk[X^*]{\{x_{l_n}^*-x_{m_n}^*:n\in\en\}}\\
&\le (2C+1)\de{x_n^*}
\end{aligned}$$
and the proof is completed.

(iv) $\Rightarrow$ (v) 
Let us assume that there is $C>0$ such that $\ca[\rho^*]{x_n^*}\le C\de{x_n^*}$ for any bounded sequence $(x_n^*)$ in~$X$. Since, by \cite[Theorem~1]{wesecom},
\[\wde{x_n^*}\le 2\dh (\clud{X}{x_n^*},X^*),\]
using the assumption we get
\begin{equation*}
\wca[\rho^*]{x_n^*}\le C\wde{x_n^*}\le 2C\dh (\clud{X}{x_n^*},X^*)= 2C\wk[X^*]{\{x_n^*:n\in\en\}}
\end{equation*}
for any bounded sequence $(x_n^*)$ in $X^*$.

(v) $\Rightarrow$ (i)  Suppose that (v) holds with a constant $C>0$. We will show that (i) holds with $2C$.  Let $T$ be an operator from $X$ to a Banach space $Y$. Fix arbitrary numbers $u<\cc{T}$ and $v>\wk[X^*]{T^*}$. It suffices to show that $u\le 2Cv$.

Since $\cc{T}>u$, there is a weakly Cauchy sequence $(x_n)$ in $B_X$ with $\ca{Tx_n}>u$. Let $l_n<m_n<l_{n+1},\,n\in\en,$ be a sequence of natural numbers and $(y_n^*)$ be a sequence in $B_{Y^*}$ such that $|y_n^*(Tx_{l_n}-Tx_{m_n})|>u$ for every $n\in\en$. 
Further, using the inequality $\wk[X^*]{T^*}<v$ we get $\wk[X^*]{\{T^*y_n^*:n\in\en\}}<v$. It follows then from the assumption (v) that $\wca[\rho^*]{T^* y_n^*}<Cv$. By passing to a subsequence, if necessary, we may assume that $\ca[\rho^*]{T^* y_n^*}<Cv$.
The set $K=\{\frac{x_{l_n}-x_{m_n}}2:n\in\en\}$ is relatively weakly compact in $B_X$ and hence there is $N\in\en$ such that $q_K(T^*y_i^*-T^*y_j^*)<Cv$ for every $i,j\ge N$.
It follows that for $j\ge N$ we have
$$\begin{aligned}u&<|y_j^*(Tx_{l_j}-Tx_{m_j})|=|T^*y_j^*(x_{l_j}-x_{m_j})|
\\&\le 2 |(T^*y_j^*-T^*y_N^*)(2^{-1}(x_{l_j}-x_{m_j}))|+|T^*y_N^*(x_{l_j}-x_{m_j})|
\\&\le 2q_K(T^*y_j^*-T^*y_N^*)+|T^*y_N^*(x_{l_j}-x_{m_j})|
\\&<2Cv+|T^*y_N^*(x_{l_j}-x_{m_j})|,\end{aligned}$$
hence
$$u\le 2Cv+\limsup_{j\to\infty}|T^*y_N^*(x_{l_j}-x_{m_j})|=2Cv,$$
as the sequence $(x_{l_j}-x_{m_j})$ is weakly null.

Finally, the equivalences (i) $\Leftrightarrow$ (vi) and (ii) $\Leftrightarrow$ (vii) follow from \eqref{eq:Gant-dh}.
\end{proof}

We have included to the previous theorem also conditions (vi) and (vii) as they quantify the classical definition of the Dunford-Pettis property. However, in view of conditions (i) and (ii) and Theorem~\ref{q-DP-univ}, it is more natural to define the Dunford-Pettis property using the implication $$T^* \mbox{ is weakly compact }\Rightarrow T \mbox{ is completely continuous},$$
as this is the formulation which can be canonically quantified.

\begin{thm}\label{q-DP-dual}
Let $X$ be a Banach space. The following assertions are equivalent:
\begin{enumerate}
\item[(i)]
There is $C>0$ such that $\cc{T^\ast}\le C\wk{T}$ for any operator $T$ from a Banach space $Y$ to $X$.
\item[(ii)]
There is $C>0$ such that $\cc{T^\ast}\le C\wk{T}$ for any operator $T$ from $\ell_1$ to $X$.
\item[(iii)]
There is $C>0$ such that $\limsup|x_n^\ast(x_n)|\le C\wk{\{x_n:n\in\en\}}$ whenever $(x_n)$ is a bounded sequence in~$X$ and $(x_n^\ast)$ is a weakly null sequence in~$B_{X^\ast}$.
\item[(iv)]
There is $C>0$ such that $\ca[\rho]{x_n}\le C\de{x_n}$ for any bounded sequence $(x_n)$ in~$X$.
\item[(v)]
There is $C>0$ such that $\wca[\rho]{x_n}\le C\wk{\{x_n:n\in\en\}}$ for any bounded sequence $(x_n)$ in~$X$.
\end{enumerate}
\end{thm}

The proof is very similar to the proof of Theorem~\ref{q-DP-direct}. Anyway, for the sake of clarity we indicate its proof.

\begin{proof}
The implication (i) $\Rightarrow$ (ii) holds trivially, even with the same constant.

(ii) $\Rightarrow$ (iii)
Let us assume that there is $C>0$ such that $\cc{T^\ast}\le C\wk{T}$ for any operator $T$ from $\ell_1$ to $X$. Let $(x_n)$ be a bounded sequence in~$X$ and $(x_n^\ast)$ be a weakly null sequence in~$B_{X^\ast}$. We will show that
\[\limsup|x_n^\ast(x_n)|\le 2C\wk{\{x_n:n\in\en\}}.\]

Let us define operator $T:\ell_1\to X$ by $T(\lambda_n)=\sum_n\lambda_n x_n$. Since $Te_n=x_n$ for every $n\in\en$, where $e_n$ denotes the $n$-th basic vector in $\ell_1$, and since $(x_1^\ast,0,x_2^\ast,0,\ldots)$ is a weakly Cauchy sequence in $B_{X^\ast}$, we can write
\[
\limsup|x_n^\ast(x_n)| = \limsup|(T^\ast x_n^\ast)(e_n)| \le\limsup \|T^*x_n^*\|\le \cc{T^\ast} \le C\wk{T}.
\]
By \cite[Theorem~2]{f-krein},
\[\wk{T}\le 2\wk{\{x_n:n\in\en\}},\]
and the conclusion follows.

The implications (iii) $\Rightarrow$ (iv) and (iv) $\Rightarrow$ (v) can be proved by copying the proofs of respective implications of Theorem~\ref{q-DP-direct}, interchanging the role of $X$ and $X^*$ and replacing $\rho^*$ by $\rho$.

(v) $\Rightarrow$ (i) 
 Suppose that (v) holds with a constant $C>0$. We will show that (i) holds with $2C$.  Let $T$ be an operator from a Banach space $Y$ to $X$. Fix arbitrary numbers $u<\cc{T^*}$ and $v>\wk[X]{T}$. It suffices to show that $u\le 2Cv$.

Since $\cc{T^*}>u$, there is a weakly Cauchy sequence $(x_n^*)$ in $B_{X^*}$ such that $\ca{T^*x_n^*}>u$. Let $l_n<m_n<l_{n+1}$, $n\in\en$, be a sequence of natural numbers and $(y_n)$ be a sequence in $B_{Y}$ such that $|(T^*x_{l_n}^*-T^*x_{m_n}^*)(y_n)|>u$ for every $n\in\en$. 
Further, using the inequality $\wk[X]{T}<v$ we get $\wk[X]{\{Ty_n:n\in\en\}}<v$. It follows then from the assumption (v) that $\wca[\rho]{Ty_n}<Cv$. By passing to a subsequence, if necessary, we may assume that $\ca[\rho]{Ty_n}<Cv$.
The set $K=\{\frac{x_{l_n}^*-x_{m_n}^*}2:n\in\en\}$ is relatively weakly compact in $B_{X^*}$, and hence there is $N\in\en$ such that $q_K(Ty_i-Ty_j)<Cv$ for every $i,j\ge N$.
It follows that for $j\ge N$ we have
$$\begin{aligned}u&<|(T^*x_{l_j}^*-T^*x_{m_j}^*)(y_j)|=|(x_{l_j}^*-x_{m_j}^*)(Ty_j)|
\\&\le 2|(2^{-1}(x_{l_j}^*-x_{m_j}^*))(Ty_{j}-Ty_{N})|+|(x_{l_j}^*-x_{m_j}^*)(Ty_{N})|
\\&\le 2q_K(Ty_{j}-Ty_{N})+|(x_{l_j}^*-x_{m_j}^*)(Ty_{N})|
\\&<2Cv+|(x_{l_j}^*-x_{m_j}^*)(Ty_{N})|,\end{aligned}$$
hence
$$u\le 2Cv+\limsup_{j\to\infty}|(x_{l_j}^*-x_{m_j}^*)(Ty_{N})|=2Cv,$$
as the sequence $(x_{l_j}^*-x_{m_j}^*)$ is weakly null.
\end{proof}

\begin{definition}
We say that a Banach space $X$ has the \emph{direct quantitative Dunford-Pettis property} if $X$ satisfies the equivalent conditions of Theorem~\ref{q-DP-direct}. In case $X$ satisfies the equivalent conditions of Theorem~\ref{q-DP-dual} we say that $X$ has the \emph{dual quantitative Dunford-Pettis property}.
\end{definition}

It is clear that while Theorem~\ref{q-DP-direct} aims to quantify the classical formulation ``every weakly compact operator from $X$ into a Banach space $Y$ is completely continuous'', whereas Theorem~\ref{q-DP-dual} is a quantification of the topological characterization of the Dunford-Pettis property ``every weakly convergent sequence in~$X$ is $\rho$-convergent''. Example~\ref{exa-podrobny} below shows that these two quantifications define different classes of Banach spaces in general.

However, the two quantifications are still connected in a way.  From the characterization (iii) of the Dunford-Pettis property in Theorem~\ref{t:znama} it is obvious that if the dual space $X^*$ of a Banach space~$X$ has the Dunford-Pettis property then the space $X$ itself has the same property. The following theorem describes an analogous result for quantitative versions.

\begin{thm}
\label{dual-conn}
For any Banach space $X$, the following assertions hold:
\begin{enumerate}
\item[(a)]
If $X^*$ has the dual quantitative Dunford-Pettis property then $X$ has the direct quantitative Dunford-Pettis property.
\item[(b)]
If $X^*$ has the direct quantitative Dunford-Pettis property then $X$ has the dual quantitative Dunford-Pettis property.
\end{enumerate}
\end{thm}

\begin{remark}\label{rm-duality} The previous theorem can be stated more precisely as follows: Let $X$ be a Banach space.
\begin{enumerate}
	\item [(a')] If $X^*$ satisfies one of the conditions (i), (iii), (iv) or (v) of Theorem~\ref{q-DP-dual} with a given constant $C$, then $X$ satisfies the respective condition of Theorem~\ref{q-DP-direct} with the same constant.
	\item [(b')] If $X^*$ satisfies one of the conditions  (iii), (iv) or (v) of Theorem~\ref{q-DP-direct} with a given constant $C$, then $X$ satisfies the respective condition of Theorem~\ref{q-DP-dual} with the same constant. In case of the assertion (i), the respective condition (i) in Theorem~\ref{q-DP-dual} is satisfied with $4C$.
\end{enumerate}
\end{remark}

\begin{proof} 
The first assertion is almost obvious, it uses only the easy facts that $\cc{T}\le\cc{T^{**}}$ for each operator $T$ and that $\ca[\rho^*]{\cdot}\le\ca[\rho]{\cdot}$ on a dual space. Let us show the second assertion for the four specified cases:

(i) Let $T:Y\to X$ be a bounded operator. Using the assumption and \eqref{eq:Gant-dh} we get
$$\cc{T^*}\le C\wk[X^{**}]{T^{**}}\le 4C\wk[X]{T}.$$

(iii) Let $(x_n^*)$ be a weakly null sequence in  $B_{X^*}$ and $(x_n)$ be a bounded sequence in $X$. Then the assumption gives
$$\limsup|x_n^*(x_n)|\le C \wk[X^{**}]{\{x_n:n\in\en\}}\le C\wk{\{x_n:n\in\en\}},$$
because the inclusion $X\subset X^{**}$ yields the  second inequality.

(iv) Let $(x_n)$ be any bounded sequence in $X$. Then $\ca[\rho]{x_n}=\ca[\rho^*]{x_n}$, where the topology $\rho^*$ on the right-hand side is considered on $X^{**}$. By the assumption
we have $\ca[\rho^*]{x_n}\le C\de{x_n}$. The quantity $\de{x_n}$ does not depend on whether we consider the sequence in $X$ or in $X^{**}$. It follows that $\ca[\rho]{x_n}\le C\de{x_n}$.

(v)  Let $(x_n)$ be any bounded sequence in $X$. Then $\wca[\rho]{x_n}=\wca[\rho^*]{x_n}$ (similarly as in the previous case). By the assumption we have $$\wca[\rho^*]{x_n}\le C \wk[X^{**}]{\{x_n:n\in\en\}}.$$
We conclude by noticing that $ \wk[X^{**}]{\{x_n:n\in\en\}}\le \wk{\{x_n:n\in\en\}}$ as in (iii).
\end{proof}

Now we are going to mention which classes of Banach spaces do have quantitative Dunford-Pettis property.
To this end let us recall the classical terminology concerning $\L_p$ spaces. If $X$ and $Y$ are isomorphic Banach spaces, by $\dd(X,Y)$ we denote their Banach-Mazur distance, i.e.,
\[\dd(X,Y) = \inf\{\|T\|\|T^{-1}\| : T \textnormal{ is an invertible operator from } X \textnormal{ onto } Y\}.\]

Let $1\le p \le\infty$ and $1\le \lambda<\infty$. A Banach space $X$ is said to be an $\L_{p,\lambda}$ space if for every finite-dimensional subspace $B$ of $X$ there is a finite-dimensional subspace $C$ of $X$ such that $C\supset B$ and $\dd(C,\ell^n_p)\le\lambda$ where $n=\dim C$. 

A Banach space is said to be an $\L_p$ space, $1\le p\le\infty$, if it is an $\L_{p,\lambda}$ space for some $\lambda<\infty$.

One of our main objectives in the rest of this paper will be the proof of the following theorem.

\begin{thm}\label{l1-li}
Every $\L_1$ space and every $\L_\infty$ space has both the direct and the dual quantitative Dunford-Pettis properties.
\end{thm}

The case of $\L_\infty$ spaces follows from Theorems~\ref{t:Linftydirect} and~\ref{t:Linftydual}.
The case of $\L_1$ spaces then follows from Theorem~\ref{dual-conn} because the dual of an $\L_1$ space is an $\L_\infty$ space by \cite[p. 58]{JoLi}.

The following example shows that the Dunford-Pettis property is not automatically quantitative in either sense.

\begin{example}\label{protipriklad}
There is a Banach space $X$ with $X^*$ separable such that 
\begin{itemize}
\item
$X$ has the dual quantitative Dunford-Pettis property, but not the direct quantitative Dunford-Pettis property,
\item
$X^*$ has the direct quantitative Dunford-Pettis property, but not the dual quantitative Dunford-Pettis property.
\item $X\oplus X^*$ has the Dunford-Pettis property but not any of its two quantitative versions.
\end{itemize}
\end{example}

The example is constructed in Section~\ref{sec-example} where several more properties of this space are stated and proved.

%%%%%%%%%%%%%%%%%%%%%%%%%%%%%%%%%%%%%%%%%%%
\section{The Schur property and quantitative Dunford-Pettis properties}
\label{sec-schur}

Let us recall that a Banach space has the {\it Schur property} if any weakly convergent sequence is norm convergent. It is obvious that any Banach space $X$ with the Schur property enjoys the Dunford-Pettis property as any operator defined on $X$ is completely continuous. A well-known consequence of this observations says that a Banach space, whose dual has the Schur property, has  the Dunford-Pettis property. Moreover, such spaces enjoy also  the reciprocal Dunford-Pettis property. We will show that these results can be refined in a quantitative way. Let us start with the following easy consequence of Rosenthal's $\ell_1$-theorem.

\begin{lemma}\label{lm-nonell1} Let $X$ be a Banach space not containing an isomorphic copy of $\ell_1$. 
 Let $Y$ be any Banach space and $T:X\to Y$ be a bounded operator. Then
$$\wk[Y]{T}\le\omega(T)\le\chi(T)\le\beta(T)\le\cc{T}.$$
\end{lemma}

\begin{proof} 
 Only the last inequality requires a proof, the remaining ones follow from \eqref{eq:diagrammsets}. So, let $(x_k)$ be any sequence in $B_X$. By Rosenthal's $\ell_1$-theorem (see \cite{ros}) there is a weakly Cauchy subsequence $(x_{k_n})$. Thus
$$\wca{Tx_k}\le\ca{Tx_{k_n}}\le\cc{T},$$
hence $\beta(T)\le\cc{T}$ which we wanted to show.
\end{proof} 

In the following proposition we explicitly formulate a trivial fact on Banach spaces with the Schur property, so no proof is required.

\begin{prop}\label{schur-triv} Let $X$ be a Banach space with the Schur property. Then any bounded linear operator $T:X\to Y$ for any Banach space $Y$ is completely continuous. In particular, $X$ has the direct quantitative Dunford-Pettis property. 
\end{prop}

\begin{thm}
\label{t:qsch2}
Let $X$ be a Banach space whose dual has the Schur property. 
\begin{itemize}
	\item[(i)] Let $T:X\to Y$  be a bounded operator. Then
\begin{equation}
\label{eq:qsch7}
\wk[Y]{T}\le \omega(T)\le\chi(T)\le\cc{T}\le 2\omega(T^*)=2\chi(T^*)\le 4\chi(T).
\end{equation}
 \item[(ii)] The space $X$ has the dual quantitative Dunford-Pettis property. More precisely, any bounded sequence $(x_n)$ in $X$ satisfies $\ca[\rho]{x_n}=\de{x_n}$.
\end{itemize}
\end{thm}

\begin{proof}
(i) The first two inequalities follow from \eqref{eq:diagrammsets}, the third one follows from Lemma~\ref{lm-nonell1} as $X$ does not contain an isomorphic copy of $\ell_1$. (If $X$ contains an isomorphic copy of $\ell_1$, by \cite[Proposition 3.3]{pelc} the dual space $X^*$ contains an isomorphic
copy of $C(\{0,1\}^\en)^*$, hence also an isomorphic copy if $C([0,1])^*$. The space $C([0,1])^*$ fails 
the Schur property as it contains a copy of $L^1(0,1)$. Thus $X^*$ fails the Schur property as well.)

The fourth inequality follows from Theorem~\ref{q-DP-univ} as $X$ has the Dunford-Pettis property.
(This follows from the second assertion (ii) or by the following reasoning. If $T:X\to Y$ is weakly compact, then $T^*:Y^*\to X^*$ is weakly compact as well by the Gantmacher theorem. Since $X^*$ has the Schur property, $T^*$ is compact. By the Schauder theorem, $T$ is compact as well, hence $T$ is completely continuous.)

Further, since $X^*$ has the Schur property, $\omega(T^*)=\chi(T^*)$.

The last inequality follows from \eqref{eq:Schauder}.

(ii) Since $X^*$ has the Schur property, it has the direct quantitative Dunford-Pettis property by Proposition~\ref{schur-triv}.
Hence $X$ has the dual version due to Theorem~\ref{dual-conn}. 

Let us show the precise version. If $X^*$ has the Schur property, it satisfies the condition (i) of Theorem~\ref{q-DP-direct} with $C=0$. Therefore it satisfies the conditions (ii) and (iii) of the same theorem with $C=0$ as well, so it satisfies the condition (iv) of the mentioned theorem with $C=1$ (all the implications follows from the computation of constants within the proof). By Remark~\ref{rm-duality} we get that
$X$ satisfies the condition (iv) of Theorem~\ref{q-DP-dual} with $C=1$, i.e., $\ca[\rho]{x_n}\le \de{x_n}$ for each bounded sequence $(x_n)$ in $X$. Since the converse inequality is obvious, the proof is completed. 
\end{proof}

Let us point out that the Schur property of the dual of a Banach space $X$ implies by Theorem~\ref{t:qsch2}(i) the inequality
\begin{equation}\label{eq:eq:rdpp}
\wk[Y]{T}\le \cc{T}
\end{equation}
for any operator $T:X\to Y$ from $X$ to a Banach space $Y$. This can be considered as a quantitative strengthening of the above mentioned fact that a space, whose dual has the Schur property, possesses the reciprocal Dunford-Pettis property.

It is worth noticing that a Banach space $X$ whose dual has the Schur property need not have to possess the direct quantitative Dunford-Pettis property (see Example~\ref{exa-podrobny}).

%%%%%%%%%%%%%%%%%%%%%%%%%%%%%%%%%%%%%%%%%%%%%%%%%
\section{Measuring weak non-compactness in $L^1$ spaces}
\label{sec-L1}

The aim of this section is to show that in the spaces of the form $L^1(\mu)$  the quantities $\omega(\cdot)$
and $\wk[]{\cdot}$ are equal. This is proved first for the case of a finite measure $\mu$, then for spaces $\ell_1(\Gamma)$ and finally for a general $\sigma$-additive non-negative measure $\mu$. 

\begin{prop}\label{L1-konecne}
Let $Y=L^1(\mu)$, where $\mu$ is a finite non-negative $\sigma$-additive measure and $X$ be any Banach space containing isometrically $Y$ as a subspace. Then
\begin{equation}
		\label{eq:L1-kon-vse}
		\omega(A)=\wk{A}=\wck{A}=\inf_{c>0}\sup \left\{\int (|f|-c)^+\di\mu: f\in A\right\}
\end{equation}
for each bounded set $A\subset Y$.
\end{prop}

\begin{proof} Let $A\subset Y$ be a bounded set. Without loss of generality suppose $A\subset B_Y$. By \eqref{eq:diagrammsets} we have
$$\wck{A}\le\wk{A}\le\omega(A).$$
Further, since $\mu$ is finite, the set $B=B_{L^{\infty}(\mu)}\subset Y$ is a weakly compact subset of $X$.
Thus 
$$\omega(A)\le \inf_{c>0}\dh(A,cB).$$
It is easy to check that  
$$\dd(f,cB)=\int(|f|-c)^+\di\mu$$
for each $c>0$ and $f\in Y$. Indeed, let $f\in Y$ be arbitrary. If $g\in cB$ is arbitrary, then
$|f-g|\ge(|f|-c)^+$ almost everywhere, which yields the inequality ``$\geq$''. The converse inequality follows from the fact that the function
$$g(t)=\begin{cases} f(t) &\mbox{ if }|f(t)|\le c,\\ c\frac{f(t)}{|f(t)|}& \mbox{ if }|f(t)|>c\end{cases}$$
belongs to $cB$ and $\int|f-g|\di\mu=\int(|f|-c)^+\di\mu$. Therefore the last quantity of \eqref{eq:L1-kon-vse} is equal to $\inf_{c>0}\dh(A,cB)$. It follows that to prove \eqref{eq:L1-kon-vse}
it is enough to show that
\begin{equation}
		\label{eq:L1-1} \wck{A}\ge \inf_{c>0}\dh(A,cB).
\end{equation}

Denote the right-hand side by $d$. If $d=0$, the inequality is obvious. So suppose that $d>0$ and fix any $\varepsilon\in(0,\frac d5)$. To finish the proof we will use the following claim which is a variant of Rosenthal's subsequence splitting lemma.

\begin{claim} There are sequences $(f_k)$, $(u_k)$, $(v_k)$ and $(w_k)$ in $Y$ satisfying the following conditions.
\begin{itemize}
	\item[(a)] $f_k\in A$ and $f_k=u_k+v_k+w_k$ for $k\in\en$.
	\item[(b)] The sequence $(u_k)$ is weakly convergent.
	\item[(c)] $\|v_k\|\le 2\varepsilon$ for $k\in\en$.
	\item[(d)] $\|\sum_{j=1}^n \alpha_j w_j\|\ge (d-3\varepsilon)\sum_{j=1}^n|\alpha_j|$ whenever $n\in\en$ and $\alpha_1,\dots,\alpha_k$ are scalars.
\end{itemize}
\end{claim}

Let us first show how the proof can be finished using this claim. The claim itself will be proved
afterwards. So suppose that we have such sequences $(f_k)$, $(u_k)$, $(v_k)$ and $(w_k)$.

Take $f^{**}$ to be any weak$^*$ cluster point of the sequence $(f_{k})$. Let $(f_\tau)$ be a subnet of the sequence $(f_{k})$ which weak$^*$ converges to $f^{**}$ and $(w_\tau)$ be the corresponding subnet of the sequence $(w_{k})$. Denote the weak limit of $(u_k)$ by $u$. Take a weak$^*$ convergent subnet $(w_\nu)$ of $(w_\tau)$ and denote the weak$^*$ limit by $w^{**}$. Then $w^{**}$ is a weak$^*$ cluster point of $(w_k)$, thus $\dd(w^{**},X)\ge d-3\varepsilon$ by (d) and \cite[Lemma 5]{wesecom}. Further, $f^{**}-w^{**}-u$ is a weak$^*$ cluster point of $(v_k)$, hence   $\|f^{**}-w^{**}-u\|\le2\varepsilon$ by (c).
It follows that
$$\dd(f^{**},X)=\dd(f^{**}-u,X)\ge \dd(w^{**},X)-\|f^{**}-w^{**}-u\|\ge d-5\varepsilon.$$
So,
$$\dd(\clu{X}{f_{k_n}},X)\ge d-5\varepsilon,$$
hence
$\wck{A}\ge d-5\varepsilon$. Since $\varepsilon\in(0,\frac d5)$ is arbitrary, $\wck{A}\ge d$. This completes the proof.

It remains to prove the claim. Fix $c_1>0$ such that $\dh(A,c_1 B)<d+\varepsilon$. We will construct by induction functions $f_k\in A$ and numbers $c_k>0$ for $k\in\en$
such that $c_1$ is the number chosen above and the following conditions are satisfied:
\begin{itemize}
	\item[(i)] $\dd(f_k,c_k B)>\dh(A,c_k B)-\varepsilon$,
	\item[(ii)] $c_{k+1}>c_k$,
	\item[(iii)] $\int_E |f_j|<\frac\varepsilon{2^k}$ for $j=1,\dots,k$, whenever  $\mu(E)\le\frac1{c_{k+1}}$.
\end{itemize}
It is obvious that the inductive construction can be performed. For each $k\in\en$ set
$E_k=\{t:|f_k(t)|>c_k\}$ and define the functions $u_k$, $v_k$ and $w_k$ as follows:
\begin{itemize}
\item If $|f_k(t)|\le c_1$ then $u_k(t)=f_k(t)$, $v_k(t)=0$, $w_k(t)=0$.
\item If $|f_k(t)|\in(c_1,c_k]$ then 
\[
u_k(t)= \frac{c_1}{|f_k(t)|}f_k(t),\ v_k(t)=\left(1-\frac{c_1}{|f_k(t)|}\right)f_k(t),\ w_k(t)=0.
\]
\item If $|f_k(t)|> c_k$ then 
\[
u_k(t)= \frac{c_1}{|f_k(t)|}f_k(t),\ v_k(t)=\frac{c_k-c_1}{|f_k(t)|}f_k(t),\ w_k(t)=\left(1-\frac{c_k}{|f_k(t)|}\right)f_k(t).
\]
\end{itemize}

Then $f_k=u_k+v_k+w_k$ for each $k\in\en$. It proves the condition (a).  Further,
since $|v_k(t)+w_k(t)|=|v_k(t)|+|w_k(t)|$ for each $t$, we get $\|v_k+w_k\|=\|v_k\|+\|w_k\|$. So,
$$\begin{aligned}
\|v_k\|&=\|v_k+w_k\|-\|w_k\|=\dd(f_k,c_1 B)-\dd(f_k,c_k B)\\
&\le \dh(A,c_1 B)-\dh(A,c_k B)+\varepsilon \le d+\varepsilon -d +\varepsilon=2\varepsilon,
\end{aligned}$$
which proves (c).

We continue by showing (d). 	
So, fix $n\in\en$ and scalars $\alpha_1,\dots,\alpha_n$. Using the triangle inequality and the fact that
$w_k=0$ outside $E_k$ we get
$$\begin{aligned}\left\|\sum_{k=1}^n \alpha_k w_k\right\|& = \int \left|\sum_{k=1}^n \alpha_k w_k\right|\di\mu
\ge\sum_{j=1}^n\int_{E_j\setminus\bigcup_{j<i\le n} E_i}\left|\sum_{k=1}^n \alpha_k w_k\right|\di\mu\\
&=\sum_{j=1}^n\int_{E_j\setminus\bigcup_{j<i\le n} E_i}\left|\sum_{k=1}^j \alpha_k w_k\right|\di\mu\\
&\ge\sum_{j=1}^n\left(|\alpha_j|\int_{E_j\setminus\bigcup_{j<i\le n} E_i}|w_j|
-\sum_{k<j} |\alpha_k|\int_{E_j\setminus\bigcup_{j<i\le n} E_i}|w_k|\right)
\\&\ge\sum_{j=1}^n\left(|\alpha_j|\left(\int_{E_j}|w_j|\di\mu-\sum_{i=j+1}^n\int_{E_i}|w_j|\di\mu\right)
-\sum_{k<j} |\alpha_k|\int_{E_j}|w_k|\right).
\end{aligned}$$
Note that $\int_{E_j}|w_j|=\dh(f_j,c_j B)\ge d-\varepsilon$. Further,
it follows from the Chebyshev inequality that $\mu(E_k)\le\frac1{c_k}$ for each $k\in\en$ (recall that $A\subset B_Y$), so using the above condition (iii) we may continue:
$$\begin{aligned}\left\|\sum_{k=1}^n \alpha_k w_k\right\|&\ge
\sum_{j=1}^n\left(|\alpha_j|\left(d-\varepsilon-\sum_{i=j+1}^n\frac{\varepsilon}{2^i}\right)
-\sum_{k<j}|\alpha_k|\frac{\varepsilon}{2^j}\right)\\
&\ge (d-2\varepsilon)\sum_{j=1}^n|\alpha_j|-\sum_{j=1}^n\sum_{k=1}^n|\alpha_k|\frac{\varepsilon}{2^j}
\ge (d-3\varepsilon)\sum_{j=1}^n|\alpha_j|.
\end{aligned}$$

Finally, the sequence $(u_k)$ is contained in $c_1 B$ and hence it is relatively weakly compact. Therefore we can without loss of generality (up to extracting a subsequence) suppose that it weakly converges. This shows (b) and the proof is complete.
\end{proof}

In the rest of this section we will often deal with $\ell_1$-sums of Banach spaces. So, let us fix some notation. Let $X=\left(\bigoplus_{\gamma\in\Gamma} X_\gamma\right)_{\ell_1}$, where $X_\gamma$ is a Banach space for each $\gamma\in\Gamma$.

If $\gamma\in\Gamma$ is arbitrary, $P_\gamma$ denotes the canonical projection of $X$ onto $X_\gamma$. Further, if $F\subset \Gamma$ is arbitrary, $P_F$ denotes the canonical projection of $X$ onto
$\left(\bigoplus_{\gamma\in F} X_\gamma\right)_{\ell_1}$. If $F=\emptyset$, we set $P_\emptyset$ to be the projection onto $\{0\}$.

The spaces $X_\gamma$, $\gamma\in\Gamma$,
and $\left(\bigoplus_{\gamma\in F} X_\gamma\right)_{\ell_1}$, $F\subset\Gamma$, are considered canonically embedded into $X$ (other coordinates are set to be zero).

\begin{lemma}\label{lm-ell1sum} Let $X_\gamma$, $\gamma\in\Gamma$, be a family of Banach spaces and let $X=\left(\bigoplus_{\gamma\in\Gamma} X_\gamma\right)_{\ell_1}$. 
Let $A\subset X$ be a bounded set. Then the following hold:
\begin{itemize}
\item[(i)] $\wck{A}\ge \inf\{\varepsilon>0: (\exists F\subset\Gamma\mbox{ finite})(\forall x\in A)(\|P_{\Gamma\setminus F}x\|<\varepsilon)\}$.
	\item[(ii)] If $A$ is weakly compact, then for each $\varepsilon>0$ there is a finite set $F\subset \Gamma$
such that $\|P_{\Gamma\setminus F}x\|<\varepsilon$ for each $x\in A$. In particular,
 the set $C=\{\gamma\in\Gamma: P_\gamma|_A\ne 0\}$ is countable.
 \item[(iii)] If, moreover, each $X_\gamma$ is reflexive, then
 \begin{multline*}\omega(A)=\wk{A}=\wck{A}\\
 =\inf\{\varepsilon>0: (\exists F\subset\Gamma\mbox{ finite})(\forall x\in A)(\|P_{\Gamma\setminus F}x\|<\varepsilon)\}.\end{multline*}
\end{itemize}
\end{lemma}

\begin{proof} (i) Let $\theta$ denote the right-hand side. The infimum is well defined as $A$ is bounded. If $\theta=0$, the inequality is obvious. So, suppose that $\theta>0$. Fix an arbitrary $\eta\in(0,\frac\theta4)$. As $\theta+\eta>\theta$, there is a finite set $F_0\subset\Gamma$ such that $\|P_{\Gamma\setminus F_0} x\|<\theta+\eta$ for each $x\in A$. We will use the following claim.

\begin{claim} There is a sequence $(x_k)$ in $A$ such that 
$$\left\|\sum_{i=1}^n\lambda_i P_{\Gamma\setminus F_0} x_i\right\|
\ge(\theta-4\eta)\sum_{j=1}^n|\lambda_j|$$
whenever $n\in\en$ and $\lambda_1,\dots,\lambda_n$ are scalars.
\end{claim}

Let us show how to conclude the proof using this claim. Let $(x_k)$ be the sequence provided by the claim. Let $x^{**}$ be any weak$^*$ cluster point of $(x_k)$ in $X^{**}$. Since $X=P_{F_0}X\oplus_1 P_{\Gamma\setminus F_0}X$ we get
$$X^{**}=(P_{F_0}X)^{**}\oplus_1 (P_{\Gamma\setminus F_0}X)^{**}
=P_{F_0}^{**}X^{**}\oplus_1 P_{\Gamma\setminus F_0}^{**}X^{**},$$
so $y^{**}=P_{\Gamma\setminus F_0}^{**}x^{**}$ is a weak$^*$ cluster point of $(P_{\Gamma\setminus F_0}x_k)$, thus
$\dd(y^{**},X)\ge\theta-4\eta$ by \cite[Lemma 5]{wesecom}. Further, clearly $\dd(x^{**},X)\ge\dd(y^{**},X)$, thus
$$\dd(\clu{X}{x_k},X)\ge\theta-4\eta,$$
in particular, $\wck{A}\ge\theta-4\eta$. As $\eta>0$ is arbitrary, we get $\wck{A}\ge\theta$ which was to be proven.

It remains to prove the claim. We will construct by induction elements $x_k\in A$ and finite sets 
$F_k\subset \Gamma$ for $k\in\en$ such that 
\begin{itemize}
	\item $\|P_{\Gamma\setminus F_{k-1}} x_k\|> \theta-\eta$,
	\item $F_k\supset F_{k-1}$,
	\item $\|P_{\Gamma\setminus F_k} x_i\|<\eta$ for $i\le k$,
	\item $\|P_{F_k\setminus F_{k-1}} x_k\|>\theta-\eta$.
\end{itemize}
The construction is easy: Recall that we have the set $F_0$. Given $F_{k-1}$, we can find $x_k$ fulfilling the first condition as $\theta-\eta<\theta$. Further, we can find a finite set $F_{k}$ satisfying the other three conditions using the properties of the $\ell_1$-sum. 

Let us show that this sequence $(x_k)$ has the required property.
%Let $y_k=P_{\Gamma\setminus F_1} x_k$. 
Let $n\ge 1$ be arbitrary and $\lambda_1,\dots,\lambda_n$ be arbitrary scalars. Then
$$\begin{aligned}
\left\|\sum_{i=1}^n\lambda_i P_{\Gamma\setminus F_0} x_i\right\|
&\ge \sum_{j=1}^n\left\|P_{F_{j}\setminus F_{j-1}}\left(\sum_{i=1}^n\lambda_i x_i\right)\right\|\\
&\ge \left.\sum_{j=1}^n\right(|\lambda_j|\|P_{F_{j}\setminus F_{j-1}}(x_j)\|
\\&\left.\qquad\qquad
-\sum_{i=1}^{j-1}|\lambda_i|\|P_{F_{j}\setminus F_{j-1}}(x_i)\|
-\sum_{i=j+1}^{n}|\lambda_i|\|P_{F_{j}\setminus F_{j-1}}(x_i)\|\right)
\\&=\sum_{j=1}^n|\lambda_j|\|P_{F_{j}\setminus F_{j-1}}(x_j)\| 
\\&\qquad\qquad
-\sum_{i=1}^n \sum_{j=i+1}^n |\lambda_i|\|P_{F_j \setminus F_{j-1}}x_i\|
-\sum_{i=1}^n \sum_{j=1}^{i-1}|\lambda_i|\|P_{F_j \setminus F_{j-1}}\|
\\&= \sum_{j=1}^n|\lambda_j|\|P_{F_{j}\setminus F_{j-1}}(x_j)\|
\\&\qquad\qquad
-\sum_{i=1}^n |\lambda_i|\left(\|P_{\Gamma \setminus F_0}x_i\|-\|P_{\Gamma\setminus F_{i-1}}x_i\|+\|P_{F_n\setminus F_i}x_i\|\right)
\\& \ge (\theta-\eta)\sum_{j=1}^n|\lambda_j| - (\theta+\eta-(\theta-\eta)+\eta)\sum_{i=1}^n|\lambda_i|
\\&=(\theta-4\eta)\sum_{j=1}^n|\lambda_j|.
\end{aligned}$$
This completes the proof of the claim and hence also (i) is proved.

(ii) The first assertion follows easily from (i). Indeed, if $A$ is weakly compact, then $\wck{A}=0$ and so the infimum is zero as well. To show the second assertion choose $F_n\subset\Gamma$  a finite set corresponding to $\varepsilon=\frac1n$. Then $C\subset\bigcup_{n\in\en} F_n$, hence it is countable.

(iii) Denote the last quantity by $\theta$. Due to (i) and \eqref{eq:diagrammsets} it is enough to prove $\omega(A)\le \theta$. Let $\varepsilon>0$ be arbitrary. Then there is a finite set $F\subset\Gamma$ such that $\|P_{\Gamma\setminus F}x\|<\theta+\varepsilon$ for each $x\in A$. Set $A_F=P_F(A)$. Then $A_F$ is a bounded subset of the reflexive space $P_F(X)$, hence it is relatively weakly compact. Therefore, $\omega(A)\le\dh(A,A_F)\le\theta+\varepsilon$ since, for any $x\in A$,
$$\dd(x,A_F)\le\|x-P_Fx\|=\|P_{\Gamma\setminus F}x\|<\theta+\varepsilon.$$
Since $\varepsilon>0$ is arbitrary, we get the sought inequality $\omega(A)\le\theta$.
\end{proof}

As an immediate consequence of Lemma~\ref{lm-ell1sum}(iii) we get the following proposition.

\begin{prop}\label{ell1Gamma} Let $X=\ell_1(\Gamma)$ for an arbitrary set $\Gamma$ and $A\subset X$ be a bounded set. Then 
$$
\chi(A)=\omega(A)=\wk{A}=\wck{A} = \inf\left\{ \sup_{x\in A} \sum_{\gamma\in\Gamma\setminus F} |x_\gamma|: F\subset\Gamma\mbox{ finite}\right\}.$$
\end{prop}

The following two lemmata extend Proposition~\ref{L1-konecne} for an arbitrary measure $\mu$. In the first one
we prove a formula for $\omega(A)$.

\begin{lemma}\label{L1-omega} Let $X=L^1(\mu)$, where $\mu$ is an arbitrary non-negative $\sigma$-additive measure and $A\subset X$ be a bounded set. Then
$$\omega(A)=\inf\left\{ \sup_{f\in A}\int(|f|-c\bchi_E)^+ \di\mu: c>0, \mu(E)<\infty\right\}.$$
\end{lemma}

\begin{proof} We start by proving the inequality `$\le$'. To do that we fix $c>0$ and a measurable set $E$ of finite measure. Let $K=\{g\in X: |g|\le c\bchi_E\ \mu\mbox{-a.e.}\}$. Then $K$ is weakly compact. Let $f\in X$ be arbitrary. Then clearly $\dd(f,K)=\int(|f|-c\bchi_E)^+ \di\mu$. Indeed, for each $g\in K$ we have $|f-g|\ge(|f|-c\bchi_E)^+$ $\mu$-a.e. and the function $g$ defined by
$$g(t)=\begin{cases} f(t) &\mbox{ if }|f(t)|\le c\bchi_{E}(t),\\ c\frac{f(t)}{|f(t)|}& \mbox{ if }|f(t)|>c\bchi_E(t)\end{cases}$$
belongs to $K$ and $\|f-g\|=\int(|f|-c\bchi_E)^+ \di\mu$.
It follows that $$\dh(A,K)=\sup_{f\in A}\int(|f|-c\bchi_E)^+ \di\mu$$ and the inequality `$\le$' is proved.

Before proving the converse inequality observe that without loss of generality we can suppose that $\mu$ is semifinite, i.e., for each measurable set $E$ with $\mu(E)>0$ there is a measurable set $E'\subset E$ with $0<\mu(E')<\infty$. Indeed, any $\mu$ can be canonically expressed as $\mu=\mu_1+\mu_2$ where $\mu_1$ is semifinite and $\mu_2$ takes only values $0$ and $\infty$ (see, e.g., \cite[Section 5]{valexa}). Moreover, this canonical decomposition fulfils the following property:
$$\forall E,\mu_1(E)<\infty \;\exists E'\subset E: \mu(E')=\mu_1(E')=\mu_1(E).$$
Then $L^1(\mu)$ is canonically isometric to $L^1(\mu_1)$ and the quantity on the right-hand side is the same for $\mu$ and $\mu_1$.

So, suppose that $\mu$ is semifinite. Let $(E_\gamma)_{\gamma\in\Gamma}$ be a maximal family of measurable sets satisfying the following conditions:
\begin{itemize}
	\item $0<\mu(E_\gamma)<\infty$ for each $\gamma\in\Gamma$,
	\item $\mu(E_\gamma\cap E_{\gamma'})=\emptyset$ for distinct $\gamma,\gamma'\in\Gamma$.
\end{itemize}
Let $\mu_\gamma$ be the restriction of $\mu$ to $E_\gamma$, i.e., $\mu_\gamma(E)=\mu(E\cap E_\gamma)$.
Then   $(\mu_\gamma)_{\gamma\in \Gamma}$ are mutually singular finite measures such that $\mu=\sum_{\gamma\in\Gamma}\mu_\gamma$. Then $L^1(\mu)$ is canonically isometric to the $\ell_1$-sum of the spaces $L^1(\mu_\gamma)$ for $\gamma\in\Gamma$ (cf. \cite[Proof of Theorem 5.1]{valexa}).

Now we are ready to  show the inequality `$\ge$'. Let $\varepsilon>0$ be arbitrary. Then there is a weakly compact set $K\subset X$ with $\dh(A,K)<\omega(A)+\varepsilon$. By Lemma~\ref{lm-ell1sum} there is $F\subset \Gamma$ finite such that
for each $f\in K$ we have
$$\int |f|(1-\bchi_{\bigcup_{\gamma\in F} E_\gamma})\di\mu<\varepsilon.$$  
Set $E_F=\bigcup_{\gamma\in F} E_\gamma$, $\mu_F=\sum_{\gamma\in F}\mu_\gamma$ and $K_F=\{f\bchi_{E_F}: f\in K\}$.
Then $K_F$ is weakly compact in $L^1(\mu_F)$. By \eqref{eq:L1-kon-vse} we obtain $c>0$ such that
$$\sup\left\{\int(|f|-c)^+\di\mu_F: f\in K_F\right\}<\varepsilon.$$

Fix $f\in A$ arbitrary. Then $\dd(f,K)<\omega(A)+\varepsilon$, so there is $g\in K$ with $\|f-g\|<\omega(A)+\varepsilon$.
Then
$$\begin{aligned}\int (|f|-c\bchi_{E_F})^+\di\mu&\le\int |f-g|\di\mu + \int (|g|-c\bchi_{E_F})^+\di\mu
\\&<\omega(A)+\varepsilon + \int_{E_F} (|g|-c)^+\di\mu_F + \int |g|(1-\bchi_{E_F})\di\mu
\\&<\omega(A)+3\varepsilon.\end{aligned}$$
Thus
$$\sup_{f\in A}\int (|f|-c\bchi_{E_F})^+\di\mu\le\omega(A)+3\varepsilon.$$
As $\varepsilon>0$ is arbitrary, we get the inequality `$\ge$'.
\end{proof}

The last result of this section finishes the extension of Proposition~\ref{L1-konecne} to arbitrary $\mu$.

\begin{thm}\label{main-L1}	Let $X=L^1(\mu)$, where $\mu$ is an arbitrary non-negative $\sigma$-additive measure and $A\subset X$ be a bounded set. Then $\omega(A)=\wk{A}=\wck{A}$.
\end{thm}

\begin{proof}  Let $A\subset L^1(\mu)$ be a bounded set. It is enough to prove that $\wck{A}\ge\omega(A)$. This will be done using Proposition~\ref{L1-konecne}, Lemma~\ref{lm-ell1sum}
(or, more exactly, claims in the respective proofs) and the formula from Lemma~\ref{L1-omega}.
We will proceed in several steps.

\medskip

{\sc Step 1:} There is a sequence $(f_k)$ in $A$ such that for each subsequence $(f_{k_n})$ we have
$\omega(\{f_{k_n}:n\in\en\})=\omega(A)$.

\smallskip

For each $f\in L^1(\mu)$ set $E_n(f)=\{t:|f(t)|>\frac1n\}$. 
Let us remark that all the sets $E_n(f)$ have obviously finite measure.

By induction we will construct for each $k\in\en$ a function $f_k\in A$ and a set $E_k$ of finite measure.

We start by choosing $f_1\in A$ such that $\int |f_1|\di\mu>\omega(A)-1$. This is possible by Lemma~\ref{L1-omega}.

Having constructed $f_1,\dots,f_k$, set
$E_k=E_k(f_1)\cup\dots\cup E_k(f_k)$. Then $E_k$ is a set of finite measure and hence there is 
some $f_{k+1}\in A$ such that 
$$\int (|f_{k+1}|-k\bchi_{E_k})^+\di\mu>\omega(A)-\frac1{k+1}.$$
This is possible again due to Lemma~\ref{L1-omega}.

This completes the inductive construction. We claim that the sequence $(f_k)$ has the required properties. This will be done using Lemma~\ref{L1-omega}.

Set $E_\infty=\bigcup_{k\in\en} E_k$. Then all the functions $f_k$ are equal to zero outside $E_\infty$.
Let $E$ be a set of finite measure and $c>0$ be arbitrary.
Fix an arbitrary $\varepsilon>0$. We can find $n\in\en$ such that $n\ge c$, $\frac{1}{n}<\frac{\varepsilon}{2}$ and $\mu((E\cap E_\infty)\setminus E_n)<\frac{\varepsilon}{2c}$.
Then for each $k\ge n$ we have
$$\begin{aligned}
\int(|f_{k+1}|&-c\bchi_E)^+ \di\mu
= \int_{E_\infty} (|f_{k+1}|-c\bchi_E)^+ \di\mu
\\&=  \int_{E_k} (|f_{k+1}|-c\bchi_E)^+ \di\mu +  \int_{E_\infty\setminus E_k} (|f_{k+1}|-c\bchi_E)^+ \di\mu
\\&\ge \int_{E_k} (|f_{k+1}|-c)^+ \di\mu +  \int_{E_\infty\setminus E_k} (|f_{k+1}|-c\bchi_E)^+ \di\mu
\\&=  \int_{E_\infty} (|f_{k+1}|-c\bchi_{E_k})^+ \di\mu - \int_{E_\infty\setminus E_k}|f_{k+1}|\di\mu    \\&\qquad\qquad+\int_{E_\infty\setminus E_k} (|f_{k+1}|-c\bchi_E)^+ \di\mu
\\&\ge\int_{E_\infty} (|f_{k+1}|-k\bchi_{E_k})^+ \di\mu - \int_{E_\infty\setminus E_k}(|f_{k+1}|- (|f_{k+1}|-c\bchi_E)^+ )\di\mu
\\&\ge\omega(A)-\frac{1}{k+1}-\int_{(E\cap E_\infty)\setminus E_k}(|f_{k+1}|- (|f_{k+1}|-c)^+ )\di\mu
\\&\ge\omega(A)-\frac1{k+1}-c\mu((E\cap E_\infty)\setminus E_k)>\omega(A)-\varepsilon.
\end{aligned}$$

This completes the proof of Step 1. Indeed, let $(f_{k_n})$ be a subsequence of $(f_k)$. Let $E$ be a set of finite measure, $c>0$ and $\varepsilon>0$. By the previous paragraph,
\[
\int(|f_{k_n}|-c\bchi_E)^+ \di\mu>\omega(A)-\varepsilon
\]
for $k_n$ large enough.
Hence 
\[
\omega(\{f_{k_n}:n\in\en\})\ge\omega(A)-\varepsilon
\]
by Lemma~\ref{L1-omega}. Since $\varepsilon>0$ is arbitrary, we get $\omega(\{f_{k_n}:n\in\en\})\ge\omega(A)$. The converse inequality is obvious.

\medskip

{\sc Step 2.} Let $A_0=\{f_k:k\in\en\}$, where $(f_k)$ is the sequence from Step 1. 
Set $$\theta=\inf\left\{\varepsilon>0:(\exists E,\ \mu(E)<\infty)(\forall f\in A_0)(\int |f|(1-\bchi_E)\di\mu<\varepsilon)\right\}.$$
By Lemma~\ref{lm-ell1sum} we get $\wck{A_0}\ge\theta$. (Indeed, let $E_\gamma$ and $\mu_\gamma$ be as in the proof of Lemma~\ref{L1-omega}. Then $\theta$ is not greater then the quantity from Lemma~\ref{lm-ell1sum}).
In particular, we have $\theta\le\omega(A_0)$ and, if $\theta=\omega(A_0)$, then $\wck{A_0}=\omega(A_0)$ and hence $\wck{A}\ge\wck{A_0}=\omega(A_0)=\omega(A)$ and the proof is finished.

So suppose that $\theta<\omega(A_0)$ and fix an arbitrary $\varepsilon\in(0,\frac16(\omega(A_0)-\theta))$.
By the definition of $\theta$ we can find $E_0$ with $\mu(E_0)<\infty$ such that for all $f\in A_0$ we have $\int |f|(1-\bchi_{E_0})\di\mu<\theta+\varepsilon$.

\medskip

{\sc Step 3.} Let $E_\gamma$ and $\mu_\gamma$ be as in Lemma~\ref{L1-omega} such that there is $\gamma_0\in\Gamma$ with $E_{\gamma_0}=E_0$. Let $\mu_0$ denote the restriction of the measure $\mu$ to $E_0$. By the claim in the proof of Lemma~\ref{lm-ell1sum}(i), there is a subsequence $(f_{k_n})$ of $(f_k)$ such that
$$\left\|\sum_{j=1}^n \lambda_j f_{k_j}(1-\bchi_{E_0})\right\|\ge(\theta-4\varepsilon)\sum_{j=1}^n|\lambda_j|$$
for each $n\in\en$ and any choice of scalars $\lambda_1,\dots,\lambda_n$.

\medskip

{\sc Step 4.} Set $A_1=\{f_{k_n}:n\in\en\}$. By Step 1 we have $\omega(A_1)=\omega(A_0)=\omega(A)$.
Further set $A_2=\{f_{k_n}\bchi_{E_0}: n\in\en\}$. Then $\omega(A_2)\ge\omega(A)-\theta-\varepsilon$.

Indeed, it follows from Lemma~\ref{L1-omega} that for each $c>0$ and $\delta>0$ there is $n\in\en$ with
$\int(|f_{k_n}|-c\bchi_{E_0})^+\di\mu>\omega(A)-\delta$. Then
$$\begin{aligned}\int(|f_{k_n}|\bchi_{E_0}-c\bchi_{E_0})^+\di\mu_0
&=\int(|f_{k_n}|-c\bchi_{E_0})^+\di\mu-\int |f_{k_n}|(1-\bchi_{E_0})\di\mu
\\&>\omega(A)-\delta-\theta-\varepsilon.
\end{aligned}$$
So, $\omega(A_2)\ge\omega(A)-\delta-\theta-\varepsilon$ by Proposition~\ref{L1-konecne}.
Since $\delta>0$ is arbitrary, $\omega(A_2)\ge\omega(A)-\theta-\varepsilon$.

\medskip

{\sc Step 5.} By the claim in the proof of Proposition~\ref{L1-konecne} there is a subsequence
$(f_{k_{n_j}})$ and sequences $(u_j)$, $(v_j)$ and $(w_j)$ in $L^1(\mu_0)\subset X$ such that 
\begin{itemize}
	\item $f_{k_{n_j}}\bchi_{E_0}=u_j+v_j+w_j$ for $j\in\en$,
	\item $(u_j)$ is weakly convergent,
	\item $\|v_j\|_X\le 2\varepsilon$ for $j\in\en$,
	\item $\|\sum_{j=1}^n\lambda_jw_j\|_X\ge(\omega(A)-\theta-4\varepsilon)\sum_{j=1}^n|\lambda_j|$, $n\in\en$ and $\lambda_1,\dots,\lambda_n$ are scalars.
\end{itemize}

\medskip

{\sc Step 6. Conclusion.} We have 
$$f_{k_{n_j}}=u_j+v_j+w_j+f_{k_{n_j}}(1-\bchi_{E_0})$$
for each $j\in\en$. 
Further,
$$\begin{aligned}
\left\|\sum_{j=1}^n\lambda_j(w_j+(1-\bchi_{E_0}) f_{k_{n_j}})\right\|
&=\left\|\sum_{j=1}^n\lambda_jw_j\right\|+\left\|\sum_{j=1}^n\lambda_j(1-\bchi_{E_0}) f_{k_{n_j}}\right\|\\&\ge (\omega(A)-8\varepsilon)\sum_{j=1}^n|\lambda_j|\end{aligned}$$
for arbitrary scalars $\lambda_1,\dots,\lambda_n$ and $n\in\en$.

Now, in the same way as in the proof of Proposition~\ref{L1-konecne} we can show that $\dd(f^{**},X)>\omega(A)-10\varepsilon$ whenever $f^{**}$ is a weak$^*$ cluster point of $(f_{k_{n_j}})$. 
It follows that $\wck{A}\ge \omega(A)-10\varepsilon$. Since $\varepsilon>0$ is arbitrary,
this completes the proof.
\end{proof}

We remind that the quantity $\omega(A)$ can be explicitly computed, see Lemma~\ref{L1-omega} for the general case and Proposition~\ref{L1-konecne} for the case of finite $\mu$. 

\begin{cor}
\label{main-L1-cor}
Every $L^1(\mu)$ space, where $\mu$ is an arbitrary non-negative $\sigma$-additive measure, has the dual quantitative Dunford-Pettis property.
\end{cor}
\begin{proof}
The fact that $L^1(\mu)$ spaces have the Dunford-Pettis property, assertion (vii) of Theorem~\ref{q-DP-univ} and Theorem~\ref{main-L1} immediately imply condition (v) in Theorem~\ref{q-DP-dual}.
\end{proof}

%%%%%%%%%%%%%%%%%%%%%%%%%%%%%%%%%%%%%%%%%%%%%
\section{Direct quantification for $C(K)$ spaces}
\label{sec-directCK}

In this section we prove that $\L_\infty$ spaces possess the direct quantitative Dunford-Pettis property.
Using the results of the previous section we prove exact results for $C(K)$ spaces (or, more generally, for $L^1$ preduals) and for preduals of $\ell_1(\Gamma)$. At the end of this section we transfer these properties to $\L_\infty$ spaces.

\begin{thm}\label{main-L1predual} Let $X$ be an $L^1$ predual, i.e., a Banach space such that $X^*$ is isometric to $L^1(\mu)$ for a non-negative $\sigma$-additive measure $\mu$. In particular, $X$ can be the space $C_0(\Omega)$ for a locally compact Hausdorff space $\Omega$, or the space $A(K)$ of continuous affine functions on a Choquet simplex~$K$. Let $Y$ be any Banach space and $T:X\to Y$ a bounded linear operator. Then
\begin{gather*}
	\wk[Y]{T}\le2\wk[X^*]{T^*}\le2\omega(T^*)=2\wk[X^*]{T^*}\le4\wk[Y]{T}\le4\omega(T),\\
	\cc{T}\le 2\omega(T^*)=2\wk[X^*]{T^*}.
\end{gather*}
\end{thm}

The first line of inequalities follows from \eqref{eq:Gant-dh}, \eqref{eq:diagrammsets} and Theorem~\ref{main-L1}. It shows the equivalence of quantities $\wk[Y]{T}$, $\wk[X^*]{T^*}$ and $\omega(T^*)$. In \cite[Example 3.2]{ka-spu-stu} we show that the quantity $\omega(T)$ is not equivalent to the other three quantities even in the case of a $C(K)$ space.

The second line shows the direct quantitative version of the Dunford-Pettis property and follows from the first line and Theorem~\ref{q-DP-univ}(ii) using the fact that $L^1$~preduals have the Dunford-Pettis property.

We continue by a stronger version of Theorem~\ref{main-L1predual} in the special case of $X^*$ being isometric to the space $\ell_1(\Gamma)$.

\begin{thm}\label{main-ell1predual}
Let $X$ be a Banach space such that $X^*$ is isometric to $\ell_1(\Gamma)$ for a set $\Gamma$.
In particular, $X$ can be the space $C(K)$ for $K$ scattered compact space or the space $c_0(\Gamma)$.
Let $Y$ be any Banach space and $T:X\to Y$ a bounded linear operator. Then the following inequalities hold.
\begin{multline*}
\wk[Y]{T}\le\omega(T)\le\chi(T)\le\cc{T}\\
\le 2\omega(T^*)=2\chi(T^*)=2\wk[X^*]{T^*}\le 4\wk[Y]{T}.
\end{multline*}
\end{thm}

The theorem follows from Theorem~\ref{t:qsch2} and Proposition~\ref{ell1Gamma}, and shows that in this case weakly compact operators, completely continuous operators and compact operators coincide and, moreover, all the quantities measuring non-compactness, weak non-compactness and non-complete continuity are equivalent. So, the spaces satisfying the assumptions of Theorem~\ref{main-ell1predual} have both the direct quantitative Dunford-Pettis property and the quantitative reciprocal Dunford-Pettis property. Eventhough we did not define the quantitative reciprocal Dunford-Pettis property, we consider the inequality $\wk[Y]{T}\le C\cc{T}$ to be an acceptable candidate because it quantifies the fact that any completely continuous operator is weakly compact (see also remarks after Theorem~\ref{t:qsch2}).

It is natural to ask whether such an inequality can be proved for general $L_1$ preduals.
It is proved in \cite{ka-spu-stu} that this is the case for $C(K)$ spaces. 
More precisely, if $X$ is a $C(K)$ space, Theorem 3.1 of \cite{ka-spu-stu} shows that, for any Banach space $Y$ and an operator $T:X\to Y$, it holds $\frac{1}{4\pi} \wk[Y]{T}\le \cc{T}\le 4\wk[Y]{T}$. On the other hand, an example is presented in \cite{ka-spu-stu} showing that $\cc{\cdot}$ is not equivalent to $\omega(\cdot)$ for operators on $C(K)$ spaces.

Finally, the last theorem of this section proves the direct quantitative Dunford-Pettis property for every $\L_\infty$ space in general. We will use the following easy proposition.

\begin{prop}
\label{qdp-iq}
Let $X$ and $Y$ be Banach spaces such that $Y$ is isomorphic to a complemented subspace of $X$.
If $X$ has either version of the quantitative Dunford-Pettis property then $Y$ has the same version of the quantitative Dunford-Pettis property.
\end{prop}
\begin{proof} It is obvious that both versions of the quantitative Dunford-Pettis properties are preserved by isomorphisms (only the respective constants may change). So, suppose that $Y$ is a complemented subspace of $X$. Let $Q$ be a bounded linear projection of $X$ onto $Y$. 

Suppose first that $X$ has the direct quantitative Dunford-Pettis property, i.e., there is $C>0$ such that $\cc{T}\le C\wk[Z]{T}$ whenever $T:X\to Z$ is an operator and $Z$ is a Banach space. 
To show that $Y$ has the same property, fix any Banach space $Z$ and an operator $T:Y\to Z$.
Since $B_Y\subset Q(B_X)\subset \|Q\| B_Y$, we have
$$
\cc{T}\le \cc{TQ}\le C\wk[Z]{TQ}\le C\|Q\|\wk[Z]{T} 
$$
and we are done.

Now suppose that $X$ has the dual quantitative Dunford-Pettis property, i.e., there is $C>0$ such that $\ca[\rho_X]{x_n}\le C\de{x_n}$ for each bounded sequence $(x_n)$ in $X$. So, let $(x_n)$ be a bounded sequence in $Y$. Then $\de{x_n}$ is the same when considered with respect to $X$ or with respect to $Y$.
Further, $Q^*$ is an isomorphic embedding of $Y^*$ into $X^*$, in particular $Q^*(B_{Y^*})\subset \|Q\| B_{X^*}$, so $\ca[\rho_Y]{x_n}\le \|Q\|\ca[\rho_X]{x_n}$. It follows
that $\ca[\rho_Y]{x_n}\le C\|Q\|\de{x_n}$ and the proof is completed.
\end{proof}

\begin{thm}\label{t:Linftydirect}
Every $\L_\infty$ space $X$ has the direct quantitative Dunford-Pettis property.
\end{thm}
\begin{proof}
By \cite[pp. 57--58]{JoLi}, $X^\ast$ is isomorphic to a complemented subspace of some $L^1(\mu)$ space $Y$. By Corollary~\ref{main-L1-cor}, $Y$ has the dual quantitative Dunford-Pettis property. Therefore, by Proposition~\ref{qdp-iq}, $X^\ast$ also has the dual quantitative Dunford-Pettis property. Consequently, using Theorem~\ref{dual-conn}(b), $X$ has the direct quantitative Dunford-Pettis property.
\end{proof}

\begin{cor}
Every $\L_1$ space has the dual quantitative Dunford-Pettis property.
\end{cor}
\begin{proof}
This follows from Theorem~\ref{dual-conn}(a) and the fact that the dual of every $\L_1$ space is an $\L_\infty$ space, see \cite[p. 58]{JoLi}.
\end{proof}

%%%%%%%%%%%%%%%%%%%%%%%%%%%%%%%%%%%%%%%%%%%%%
\section{Dual quantification for $C(K)$ spaces}
\label{sec-dualCK}

In this section we show that $\L_\infty$ spaces enjoy the dual quantitative Dunford-Pettis property.
The first step is again an exact result on $C(K)$ spaces.

We start by the following proposition. Its first part is a quantification of the fact that in $C(K)$ any bounded pointwise convergent sequence is weakly convergent. The second part is a quantitative version of the Egoroff theorem.

\begin{prop}
\label{egoroff} Let $K$ be a compact space and let $(f_n)$ be a bounded sequence of continuous functions on $K$. Then the following assertions hold.
\begin{itemize}
	\item[(i)]  $\displaystyle\de{f_n}=\sup_{x\in K}\inf_{n\in\en} \sup_{i,j\ge n} |f_i(x)-f_j(x)|$.
	\item[(ii)] Let $\mu$ be a positive Radon measure on $K$. 
	Then for any $\varepsilon>0$ there exists a compact set $L\subset K$ such that $\mu(K\setminus L)\le\varepsilon$ and $\ca{f_n|_L}\le\de{f_n}$, where the sequence of functions $(f_n|_L)$ is considered in $C(L)$.
\end{itemize}
\end{prop}
\begin{proof}
(i) The inequality `$\ge$' is obvious. Let us prove the converse one. Denote by $c$ the quantity on the right-hand side. For $n\in\en$ we define the function
$$g_n(x)=\sup_{i,j\ge n} |f_i(x)-f_j(x)|,\qquad x\in K.$$
Then $g_n$ is a non-negative lower semicontinuous (and hence Borel) function on $K$. Let $\nu\in B_{C(K)^*}$ be arbitrary. By the Riesz representation theorem we identify $\nu$ with a signed or complex Radon measure on $K$. Then
$$\begin{aligned}\inf_{n\in\en}\sup_{i,j\ge n} \left| \int (f_i-f_j)\di\nu\right|
&\le \inf_{n\in\en}\sup_{i,j\ge n} \int |f_i-f_j|\di|\nu|
\le \inf_{n\in\en} \int  g_n \di|\nu|
\\& = \int \inf_{n\in\en} g_n \di|\nu|
\le  \int  c \di|\nu| \le c.\end{aligned}$$
The only equality in this computation follows from the monotone convergence theorem, all the inequalities are trivial. Since $\de{f_n}$ is the supremum of the quantities on the left-hand side over $\mu\in B_{C(K)^*}$, we get $\de{f_n}\le c$ and conclude the proof.

(ii) For any two natural numbers $m$ and $k$ we define 
$$Q_{m,k} =\left\{ x\in K: \sup_{i,j\ge m}|f_i(x)-f_j(x)|>\de{f_n}+\frac1k\right\}.$$
The sets $Q_{m,k}$ are open in $K$, $Q_{m+1,k}\subset Q_{m,k}$, and $\bigcap_m Q_{m,k}=\emptyset$. It follows that $\mu(Q_{m,k})\to 0$ as $m\to\infty$. One can therefore choose $m_k$ so that $\mu(Q_{m_k,k})\le\frac{\varepsilon}{2^k}$. If $x$ belongs to $K\setminus Q_{m_k,k}$, we have
\[|f_i(x)-f_j(x)|\le\de{f_n}+\frac1k\]
for any $i,j\ge m_k$. It suffices to take $L=K\setminus\bigcup_k Q_{m_k,k}$.
\end{proof}

We will need the following well-known characterization of weakly compact subsets of $L^1(\mu)$.

\begin{lemma}[Dunford-Pettis, see {\cite[Theorem~4.21.2]{edwards}}]
\label{lm:dp}
Let $\mu$ be a positive Radon measure on a compact space $K$. In order that a subset $P$ of $L^1(\mu)$ be relatively weakly compact, it is necessary and sufficient that the following conditions be fulfilled:
\begin{itemize}
\item 
$\sup\{\int|f|d\mu:f\in P\}<\infty$.
\item  
Given $\varepsilon>0$, there exists a number $\delta>0$ such that 
\[\sup\left\{\int_A|f|d\mu:f\in P\right\}\le\varepsilon\]
provided $A\subset K$ is measurable and $\mu(A)\le\delta$.
\end{itemize}
\end{lemma}

The following lemma is the key step to prove the dual quantitative Dunford-Pettis property of $C(K)$ spaces.

\begin{lemma}
\label{q_P<de}
Let $K$ be  a compact  space and $\mu$ be a positive Radon measure on $K$.
Consider $L^1(\mu)$ canonically embedded into $C(K)^*$. Then for any bounded sequence $(f_n)$ in $C(K)$  and any relatively weakly compact subset $P$ of $B_{L^1(\mu)}$ we have
\[\inf_{n_0\in\en}\sup\{q_P(f_i-f_j):i,j\ge n_0\}\le\de{f_n}.\]
\end{lemma}
\begin{proof}
Without loss of generality, let us assume that $\|f_n\|\le 1$ for every $n\in\en$. 
Let $\varepsilon>0$. Using Lemma~\ref{lm:dp}, we first choose $\delta>0$ so that for any measurable set $A$ satisfying $\mu(A)\le\delta$, one has
\[\int_A|h|d\mu\le\frac{\varepsilon}{4}\]
for all $h\in P$.
By Proposition~\ref{egoroff}, we may choose a compact set $L\subset K$ such that $\mu(K\setminus L)\le\delta$ and $\ca{f_n|_L}\le\de{f_n}$. It follows that for any $h$ in $P$, $n_0\in\en$ and $i,j\ge n_0\in\en$, one has
\begin{eqnarray*}
\left|\int_K h(f_i-f_j)d\mu\right| &\le& \int_L|h(f_i-f_j)|d\mu + \int_{K\setminus L}|h(f_i-f_j)|d\mu\\
&\le& \sup_{k,l\ge n_0}\|(f_k-f_l)|_L\| + 2\cdot\frac{\varepsilon}4.
\end{eqnarray*}
Since the right-hand side tends to $\ca{f_n|_L}+\frac{\varepsilon}2$ as $n_0\to\infty$, we can determine $n_1\in\en$ independent of $h$ in $P$ such that $i,j\ge n_1$ entails
\[\left|\int h(f_i-f_j)d\mu\right|<\de{f_n}+\varepsilon\]
for all $h$ in $P$. This concludes the proof.
\end{proof}

\begin{thm}
\label{qdpt-ck}
Let $K$ be a compact space. Then for any bounded sequence $(f_n)$ in $C(K)$ we have
\[\ca[\rho]{f_n}=\de{f_n}.\]
In particular, $C(K)$ has the dual quantitative Dunford-Pettis property.
\end{thm}
\begin{proof}
It is enough to prove $\ca[\rho]{f_n}\le\de{f_n}$, since the other inequality is always true. Let $H$ be a weakly compact subset of $B_{C(K)^\ast}$. In order to establish that
\[\inf_{n_0\in\en}\sup\{q_H(f_i-f_j):i,j\ge n_0\}\le\de{f_n},\]
it suffices to prove this inequality for any countable subset of $H$. So we may assume that $H$ is countable a relatively weakly compact. 

As in the proof of \cite[Theorem~9.4.4]{edwards}, the problem is reducible to the case in which $H=\{h\cdot\mu:h\in P\}$, where $\mu$ is a certain positive Radon measure on~$K$ and $P$ is a relatively weakly compact subset of $L^1(\mu)$. Indeed, let $H=\{\mu_n:n\in\en\}$. Then $\mu=\sum\frac{|\mu_n|}{2^n}$ is a positive Radon measure on~$K$. We define $u:L^1(\mu)\to C(K)^*$ by $u(h)=h\cdot\mu$ for every $h\in L^1(\mu)$. Since each $\mu_n$ is absolutely continuous relative to $\mu$, $u(L^1(\mu))$ contains each $\mu_n$. Moreover, $u$ is an isometric isomorphism of $L^1(\mu)$ onto a closed subspace of $C(K)^*$ containing $H$.

Application of Lemma~\ref{q_P<de} now finishes the proof.
\end{proof}

\begin{cor}\label{c:L1mudirect} Let $\mu$ be a non-negative $\sigma$-additive measure. Then the space $X=L^1(\mu)$ has the direct quantitative Dunford-Pettis property. Moreover,
$$\cc{T}\le 4\wk[X^*]{T^*}$$ whenever $Y$ is a Banach space and $T:X\to Y$ an operator.
\end{cor}

\begin{proof} The space $X^*$ is a $C(K)$-space, so it is enough to use Theorem~\ref{qdpt-ck} and Theorem~\ref{dual-conn}. Let us prove the `moreover' part. By  Theorem~\ref{qdpt-ck}, the space $X^*$ satisfies the condition (iv) of Theorem~\ref{q-DP-dual} with $C=1$. By Remark~\ref{rm-duality}, the space $X$ satisfies the condition (iv) of Theorem~\ref{q-DP-direct} with $C=1$ as well. It follows from the proof of the implications (iv) $\Rightarrow$ (v) and (v) $\Rightarrow$ (i) of Theorem~\ref{q-DP-direct}
that $X$ satisfies the respective condition (i) with $C=4$. This completes the proof.
\end{proof}

\begin{thm}\label{t:Linftydual}
Every $\L_\infty$ space $X$ has the dual quantitative Dunford-Pettis property.
\end{thm}

\begin{proof}
By \cite[pp. 57--58]{JoLi}, $X^\ast$ is isomorphic to a complemented subspace of a space of the form  $L^1(\mu)$ for a non-negative $\sigma$-additive measure $\mu$. Hence $X^*$ has the direct quantitative Dunford-Pettis property  by Corollary~\ref{c:L1mudirect} and Proposition~\ref{qdp-iq}. Consequently, $X$ has the dual quantitative Dunford-Pettis property by Theorem~\ref{dual-conn}.
\end{proof}

\begin{cor}
Every $\L_1$ space has the direct quantitative Dunford-Pettis property.
\end{cor}

%%%%%%%%%%%%%%%%%%%%%%%%%%%%%%%%%%%%%%%%%%%%%%%%%%%%%

\section{An example}\label{sec-example}

In this section we present two results -- the first one is a detailed version of Example~\ref{protipriklad}; the second one compares the quantities $\wk[]{\cdot}$ and $\omega(\cdot)$ in the space $c_0(\Gamma)$. It is used to formulate the example in a more precise way, but it is simultaneously of an independent interest. 

\begin{example}\label{exa-podrobny} There is a Banach space $X$ with the following properties
\begin{itemize}
	\item[(i)] The space $X^*$ is a separable $L$-embedded space with the Schur property. In particular, $X^*$ has the direct quantitative Dunford-Pettis property and $X$ has the dual quantitative Dunford-Pettis property.
	\item[(ii)] There is a sequence $(A_n)$ of subsets of $B_{X^*}$ such that $\omega(A_n)=\chi(A_n)\ge\frac14$
	for each $n\in\en$ and $\wk[X^*]{A_n}\to0$.
	\item[(iii)] There is a sequence $(T_n)$ of bounded linear operators $T_n:X\to c_0$ such that $\|T_n\|\le 2$,  $\cc{T_n}\ge 1$ for each $n\in\en$ and $\omega(T_n)=\wk[c_0]{T_n}\to0$.
	\item[(iv)] The space $X$ does not have the direct quantitative Dunford-Pettis property and $X^*$ does not have the dual quantitative Dunford-Pettis property.
	\item[(v)] The space $X\oplus X^*$ has the Dunford-Pettis property but does not have any of the two variants of quantitative Dunford-Pettis property.
\end{itemize}
\end{example}

\begin{proof} We will construct the space $X$ and operators $T_n$ satisfying the conditions (i) and (iii). Then the assertions (iv) and (v) will be satisfied automatically. Indeed, it follows from (iii) that $X$ does not satisfy the condition (vi) of Theorem~\ref{q-DP-direct} and thus $X$ does not have the direct quantitative Dunford-Pettis property. Using Theorem~\ref{dual-conn} we then conclude that $X^*$ has not the dual quantitative Dunford-Pettis property, which completes the proof of the assertion (iv).
Further, by (i) both $X$ and $X^*$ have the Dunford-Pettis property, hence so does $X\oplus X^*$. It follows from (iv) and Proposition~\ref{qdp-iq} that $X\oplus X^*$ does not have any of the two quantitative versions of the Dunford-Pettis property.

Let us continue by describing the space $X$ and the operators $T_n$. 
Fix an arbitrary $\alpha>0$. Set $$B_\alpha=\alpha B_{c_0}+B_{\ell_1}\subset c_0.$$
Since $B_{\ell_1}$ is weakly compact in $c_0$, $B_\alpha$ is the closed unit ball of an equivalent norm on $c_0$.
Denote this space by $X_\alpha$ and the identity mapping of $X_\alpha$ onto $c_0$ by $I_\alpha$. Then $I_\alpha$ is an onto isomorphism and $\|I_\alpha\|=1+\alpha$.
So, in particular $X_\alpha$ is isomorphic to $c_0$ and hence $X_\alpha^*$ is isomorphic to $\ell_1$.
 The norm on $X_\alpha^*$ is easily computed to be given by the formula
$$\|x^*\|^*_\alpha=\alpha\|x^*\|_1+\|x^*\|_\infty.$$
Further, $X_\alpha^{**}$ is isomorphic to $\ell_\infty$ and by the Goldstine theorem the closed unit ball is equal to
$$\wscl{B_\alpha}=\alpha B_{\ell_\infty}+B_{\ell_1}.$$
The third dual $X_\alpha^{***}$ is isomorphic to $\ell_\infty^*=M(\beta\en)$, the space of all (signed or complex) Radon measures on the \v{C}ech-Stone compactification of $\en$. The norm is given by
the formula
$$\|\mu\|^{***}_\alpha=\alpha\|\mu\|_{M(\beta\en)}+\|(\mu\{k\})_{k=1}^\infty\|_\infty.$$
It follows that $X_\alpha^*$ is $L$-embedded. Indeed, the respective projection of $X_{\alpha}^{***}$ onto $X_\alpha^*$ can be defined by
$$\mu\mapsto \mu|_{\en}=(\mu\{k\})_{k=1}^\infty,\qquad \mu\in X_{\alpha}^{***}.$$
Moreover, $X_\alpha^*$ has the Schur property, as it is isomorphic to $\ell_1$.

Further, let
$$X=\left(\bigoplus_{n\in\en} X_{1/n}\right)_{c_0}.$$
Then 
$$X^*=\left(\bigoplus_{n\in\en} X_{1/n}^*\right)_{\ell_1},$$
in particular, $X^*$ is an $\ell_1$-sum of $L$-embedded separable spaces with the Schur property, thus it is a separable $L$-embedded space  (by  \cite[Proposition~1.5]{hawewe}) 
and has the Schur property as well (this follows
by a straightforward modification of the proof that $\ell_1$ has the Schur property, see \cite[Theorem~5.19]{fhhmpz}). It follows that the assertion (i) is satisfied (using, moreover, Proposition~\ref{schur-triv} and Theorem~\ref{dual-conn}).

Denote by $P_n$ the projection of $X$ onto the $n$-th coordinate and 	set $T_n=I_{1/n}P_n$. As $\|P_n\|=1$, we have $\|T_n\|\le 1+\frac1n\le2$. 

Further, fix an arbitrary $n\in\en$. 

Let $(x_k)$ be the canonical basis of $X_{1/n}$ (embedded in $X$). Then $(x_k)$ is a weakly Cauchy sequence in $B_X$ and $\ca{T_nx_k}=1$. Thus $\cc{T_n}\ge1$.

Further, $\omega(T_n)\le\frac1n$, as $T_n B_X=B_{1/n}=B_{\ell_1}+\frac1n B_{c_0}$ and $B_{\ell_1}$ is weakly compact in $c_0$. Hence $\omega(T_n)\to0$. Since $\wk[c_0]{T_n}\le \omega(T_n)$ by \eqref{eq:diagramm-operators}, we get $\wk[c_0]{T_n}\to0$ as well. That in fact $\wk[c_0]{T_n}= \omega(T_n)$ follows from Proposition~\ref{c0Gamma} below. This completes the proof of the assertion (iii).

It remains to prove the assertion (ii). To do that it is enough to set $A_n=\frac12T_n^*(B_{\ell_1})$.
To verify it let us consider the operator $T_n^*:\ell_1\to X^*$. We have $T_n^*=P_n^* I_{1/n}^*$. The operator  $P_n^*$ is the injection of $X_{1/n}^*$ into $X^*$ (made by setting other coordinates to be $0$). Further, operator $I_\alpha^*$ is the identity of $\ell_1$ onto $X_{\alpha}^*$.
In particular, let $(e_k)$ be the canonical basic sequence in $X_\alpha^*$. Then $\|e_k-e_l\|_\alpha^*=2\alpha+1$ for $k,l\in\en$ distinct, thus $\wca{e_k}>1$. In particular,
$\beta(I_\alpha^*)>1$. As $P_n^*$ is an isometric embedding, we have $\beta(T_n^*)=\beta(I_{1/n}^*)>1$.
Since $X^*$ has the Schur property, using \eqref{eq:compact} we obtain $\omega(T_n^*)=\bchi(T_n^*)>\frac12$, thus $\omega(A_n)=\bchi(A_n)>\frac14$.

Finally, using (ii) and \eqref{eq:Gant-dh} we have
$$\wk[X^*]{A_n}=\frac12\wk[X^*]{T_n^*}\le\wk[c_0]{T_n}\to 0.$$
This completes the proof. Anyway, let us estimate $\wk[X^*]{T_n^*}$ explicitly. Let us first notice that $\wk[X^*]{T_n^*}\le\wk[X_{1/n}^*]{I_{1/n}^*}$ (as $P_n^*$ is an isometric embedding).
So, let us estimate $\wk[X_\alpha^*]{I_\alpha^*}$:

We have $I_\alpha^*(B_{\ell_1})=B_{\ell_1}\subset X_\alpha^*$. By the Goldstine theorem its weak$^*$ closure in $X_\alpha^{***}$ is equal to $B_{M(\beta\en)}$. Fix any $\mu\in B_{M(\beta\en)}$. Then $\mu|_\en\in X_\alpha^*$ and 
$$\|\mu-\mu|_\en\|_\alpha^{***}=\alpha\|\mu-\mu|_\en\|_{M(\beta\en)}\le\alpha,$$
thus
$$\wk[X_\alpha^*]{I_\alpha^*}=\dh(B_{M(\beta\en)},X_\alpha^*)\le\alpha.$$
It follows that 
$$\wk[X^*]{T_n^*}\le\wk[X_{1/n}^*]{I_{1/n}^*}\le\frac1n\to0.$$
\end{proof}	

The following proposition was used in the previous example to precise the formulation. Anyway, it is of independent interest as it is a partial answer to a general open question (see the next section).

\begin{prop}\label{c0Gamma} Let $X=c_0(\Gamma)$ for a set $\Gamma$. Then $\wk{A}=\omega(A)$ for any bounded set $A\subset X$.

Moreover, if $K\subset X^{**}$ is weak$^*$ compact, then there is weakly compact set $L\subset X$ with
$\dh(K,L)=\dh(K,X)$.
\end{prop}

\begin{proof} It is enough to prove the `moreover' statement. Indeed, $\wk{A}\le\omega(A)$ by \eqref{eq:diagrammsets}. Conversely, $\wk{A}=\dh(\wscl A,X)$ and $\wscl A$ is weak$^*$ compact in $X^{**}$.
If we are able to find $L\subset X$ weakly compact such that $\dh(\wscl A,L)=\dh(\wscl A,X)$, then
$$\omega(A)\le\dh(A,L)\le\dh(\wscl A,L)=\dh(\wscl A,X)=\wk{A}.$$

So, let us prove the `moreover' statement. The space $X^{**}$ is canonically identified with $\ell_\infty(\Gamma)$ and the weak$^*$ topology on bounded sets coincides with the topology of pointwise convergence. Fix an arbitrary $c>0$ and define the mapping $\Psi_c:\ell_\infty(\Gamma)\to\ell_\infty(\Gamma)$ by the formula
$$\Psi_c(x)(\gamma)=\begin{cases} 0 & \mbox{ if }|x(\gamma)|\le c, \\
x(\gamma)(1-\frac c{|x(\gamma)|}) & \mbox{ if }|x(\gamma)|>c.
\end{cases}$$
Then $\Psi_c$ is pointwise-to-pointwise continuous. Moreover, $\|\Psi_c(x)-x\|\le c$ for each $x\in\ell_\infty(\Gamma)$ and $\Psi_c(x)\in c_0(\Gamma)$ if and only if $\dd(x,c_0(\Gamma))\le c$. 
Indeed,
$$\dd(x,c_0(\Gamma))= \inf\{ \sup_{\gamma\in\Gamma\setminus F} |x(\gamma)|: F\subset\Gamma\mbox{ finite}\}.$$
So, let $K\subset X^{**}$ be weak$^*$ compact. Set $c=\dh(K,X)$. Then $L=P_c(K)$ is contained in $X$, it is weakly compact and $\dh(K,L)\le c$. This completes the proof.
\end{proof}

%%%%%%%%%%%%%%%%%%%%%%%%%%%%%%%%%%%%%%%%%%%%%%%%%%%%%%%%%%%
\section{Open problems}

In the final section we collect some open questions which arised naturally during our research.

\begin{question}\label{question1} Let $X=C(K)$ (or, more generally, let $X$ be an $L_1$ predual). Are the quantities
$\omega(\cdot)$ and $\wk{\cdot}$ equal, or at least equivalent?
\end{question}

By Proposition~\ref{c0Gamma}, the two quantities are equal for $X=c_0(\Gamma)$. It follows that they
are equivalent for $X=C(\alpha\Gamma)$, the space of continuous functions on the one-point compactification of the discrete space $\Gamma$, as this space is isomorphic to $c_0(\Gamma)$. However, 
we do not know whether even in this easy examples the quantities are in fact equal. We also do not know
what happens for general $C(K)$ spaces, in particular for $C([0,1])$.

The fact that this question is interesting and may be rather hard is illustrated by the fact that from the positive answer it would easily follow that Eberlein compact spaces are preserved by continuous mappings. This is a well-known but nontrivial result. Let us comment this connection in a more detail. Recall that a compact space $K$ is called {\it Eberlein} if it is homeomorphic to a subset of $(X,w)$ for a Banach space $X$. 

So, suppose that the previous question has positive answer. Let $K$ be a continuous image of an Eberlein compact space. Then the space $C(K)$ is easily seen to be isomorphic to a subspace of a weakly compactly generated space. Using Theorem~\ref{t:swcg} and our assumption we get that $C(K)$ is in fact weakly compactly generated (we remark that we use only the easy implication of the second statement of Theorem~\ref{t:swcg}). Hence, $K$ is easily seen to be an Eberlein compact space.

\begin{question}\label{ques3} Let $X$ be a Banach space. Suppose that there is $C>0$ such that for each operator $T:X\to c_0$ we have $\cc{T}\le C\wk[c_0]{T}$. Does $X$ have the direct quantitative Dunford-Pettis property?
\end{question}

By Theorem~\ref{t:znama}, the space $X$ does have the Dunford-Pettis property. Further, to ensure that $X$ has the direct quantitative Dunford-Pettis property it is enough that such an inequality holds for operators from $X$ to $\ell_\infty$. It is not clear whether $\ell_\infty$ can be replaced by $c_0$. 
The space $X$ from Example~\ref{exa-podrobny} which fails the direct quantitative Dunford-Pettis property fails this property also for operators to $c_0$.

\begin{question} Let $X$ be a Banach space. Suppose that there is $C>0$ such that for each Banach space $Y$ and each operator $T:X\to Y$ we have $\cc{T}\le C\omega(T)$. Does $X$ have the direct quantitative Dunford-Pettis property?
\end{question}

The stated property is a formally weaker version of the direct quantitative Dunford-Pettis property
(see Theorem~\ref{q-DP-direct}(vi) and \eqref{eq:wcomp2}). We do not know any example showing that this property is really weaker, the space $X$ from Example~\ref{exa-podrobny} fails even the weaker version.
Let us remark that the positive answer to Question~\ref{ques3} implies the positive answer to the present question due to Proposition~\ref{c0Gamma}. Moreover, the positive answer to Question~\ref{question1} also implies the positive answer to the last question. Indeed, by Theorem~\ref{q-DP-direct} it is enough to consider operators $T:X\to\ell_\infty$ and $\ell_\infty$ is a $C(K)$ space.

\begin{question} Suppose that $X$ is a Banach space such that $X^*$ satisfies the dual quantitative Dunford-Pettis property.
\begin{itemize}
	\item[(a)] Does $X^*$ have the direct  quantitative Dunford-Pettis property?
	\item[(b)] Does $X$ have the dual  quantitative Dunford-Pettis property?
\end{itemize}
 \end{question}

It follows from Theorem~\ref{dual-conn} that the positive answer to (a) implies the positive answer to (b). Example~\ref{exa-podrobny} shows that the two versions of the quantitative Dunford-Pettis property are incomparable in general. However, it does not answer the above question. In particular, we do not know whether $X^{**}$ has the dual quantitative Dunford-Pettis property if $X$ is the space from Example~\ref{exa-podrobny}.

%%%%%%%%%%%%%%%%

%\bibliography{qdpp}\bibliographystyle{plain}
%\begin{thebibliography}{1}
\def\cprime{$'$}

\end{document}